\documentclass[11pt]{article}

\usepackage[letterpaper,margin=1in]{geometry}
\usepackage[T1]{fontenc}
\usepackage{lmodern}
\usepackage{microtype}
\usepackage{amsmath,amssymb,amsthm,mathtools}
\usepackage{booktabs,longtable,array}
\usepackage{graphicx}
\usepackage{caption}
\usepackage{enumitem}
\usepackage{xcolor}
\usepackage[numbers,sort&compress]{natbib}
\usepackage[hidelinks]{hyperref}

\newtheorem{theorem}{Theorem}[section]
\newtheorem{lemma}[theorem]{Lemma}
\newtheorem{proposition}[theorem]{Proposition}

\theoremstyle{remark}

\newcommand{\R}{\mathbb{R}}
\newcommand{\E}{\mathbb{E}}
\newcommand{\Pp}{\mathbb{P}}

\newcommand{\argmin}{\operatorname*{arg\,min}}
\newcommand{\argmax}{\operatorname*{arg\,max}}
\newcommand{\conv}{\operatorname{conv}}
\newcommand{\interior}{\operatorname{int}}
\newcommand{\barchi}[2]{\overline{\chi}^{\,2}_{#1,#2}}

\usepackage{hyperref}

\title{\textbf{Empirical likelihood confidence regions for ordered bivariate means}}
\author{
Naresh Garg\\
\small Department of Computer Science, Aalto University, Espoo, Finland\\
\small ELLIS Institute Finland, Espoo, Finland\\
\small \texttt{\href{mailto:gargnaresh22@gmail.com}{gargnaresh22@gmail.com},
\href{mailto:naresh.garg@aalto.fi}{naresh.garg@aalto.fi}}
}
\date{}

\begin{document}
\maketitle

\begin{abstract}
Let $\boldsymbol{X}_i=(X_{1i},X_{2i})^\top$ be independent and identically distributed observations with mean $\boldsymbol{\mu}=(\mu_1,\mu_2)^\top$ constrained by $\mu_1\leq\mu_2$. We study empirical-likelihood inference for a fixed mean vector and distinguish it from the previously known test of equality against an ordered alternative. At a fixed interior point, the constrained empirical likelihood ratio has the usual $\chi^2_2$ limit. At a fixed boundary point $(m,m)^\top$, its limit is the chi-bar-square distribution $\tfrac12\chi^2_1+\tfrac12\chi^2_2$. By contrast, profiling the unknown common mean in the equality-versus-order test yields $\tfrac12\chi^2_0+\tfrac12\chi^2_1$, the $k=2$ ordered-mean case of El Barmi (1996). We give an exact reduction of the latter statistic to the empirical likelihood of the paired differences, establish the localization step needed for the fixed-boundary expansion, and derive a local-to-boundary limit showing that interior calibration is not uniform over $n^{-1/2}$-neighborhoods of the boundary. Monte Carlo experiments under Gaussian, Student $t_5$, and shifted log-normal sampling examine fixed, boundary, and local regimes with explicit numerical-failure accounting. Illustrative paired-data analyses show the practical distinction between fixed-candidate confidence regions, directional equality tests, and ordinary scalar empirical-likelihood intervals truncated to the nonnegative parameter space.
\end{abstract}


\section{Introduction}

Order restrictions provide a natural way to incorporate qualitative scientific information into statistical inference. Such restrictions arise, for example, when responses are expected to increase with dose, when reliability measures deteriorate over time, when economic indicators have a prespecified ordering, or when the direction of change between paired measurements is known in advance. For finite-dimensional parameters, these assumptions commonly define convex cones or intersections of half-spaces. The classical literature has developed an extensive theory of estimation and testing in such parameter spaces, including isotonic regression, projection estimators, and likelihood-ratio tests under cone alternatives; see \citet{barlow1972statistical}, \citet{robertson1988order}, \citet{van2006restricted}, \citet{silvapulle2005} and \citet{garg2024unified}. Confidence-region construction raises an additional issue: the appropriate likelihood-ratio calibration may change when a candidate parameter lies on the boundary of the constrained space (\citet{hwang1994confidence} and \citet{park2014confidence}).

Empirical likelihood (EL) provides a nonparametric likelihood framework for addressing this problem. Rather than specifying a parametric distribution, EL assigns probability masses to the observed data subject to moment constraints (see \citet{owen1988empirical,owen1990,owen2001}). For an unconstrained $p$-dimensional mean, the empirical log-likelihood ratio evaluated at the true mean converges to $\chi^2_p$ under standard moment and nonsingularity conditions, giving a nonparametric analogue of Wilks' theorem. The methodology extends to parameters defined through general estimating equations and to additional restrictions imposed on those parameters (\citet{qinlawless1994,qin1995estimating}).

The usual Wilks calibration need not remain valid when the true parameter lies on the boundary of an inequality-constrained parameter space. At a boundary point, the local parameter space is a tangent cone rather than the full Euclidean space, and the limiting likelihood-ratio statistic is determined by projection of a Gaussian vector onto that cone (\citet{chernoff1954distribution},\citet{geyer1994asymptotics}). In parametric likelihood theory, this mechanism produces chi-bar-square distributions, namely mixtures of chi-square distributions with different degrees of freedom (see \citet{shapiro1985,selfliang1987,silvapulle2005}). The mixture components and their weights depend on the local geometry of the constraint and, in general, on the covariance structure.

The corresponding empirical-likelihood theory was developed by \citet{el1996empirical}, who established chi-bar-square limits for parameters defined through estimating equations under smooth inequality constraints. In particular, his ordered-means example considers testing the equality of $k$ means against an order-restricted alternative under a general nonsingular covariance structure. Related empirical-likelihood developments include inference for mean vectors under inequality constraints (\citet{chen2011empirical}), restricted one-way analysis of variance (\citet{wen2011restricted}), tests for stochastic ordering (\citet{elbarmi2012stochastic}), order-restricted inference with missing observations (\citet{wang2017order}), and inference under inequality constraints for dependent high-frequency data (\citet{liao2026empirical}). Against this background, the present paper specializes the theory to two ordered means and focuses on the inferential distinction between evaluating a fully specified candidate vector on the boundary and testing a composite equality hypothesis in which the common mean is profiled out. This distinction is particularly important for likelihood-ratio calibration and the construction of empirical-likelihood confidence regions.

Let
$$
\boldsymbol{\mu}=(\mu_1,\mu_2)^\top,
\qquad
\Theta=\{\boldsymbol{\mu}\in\mathbb{R}^2:\mu_1\leq\mu_2\},
$$
and let $L_n(\boldsymbol{\mu})$ denote the empirical likelihood for the bivariate mean. For confidence-region inversion, a candidate vector $\boldsymbol{\mu}_0\in\Theta$ is completely specified and is compared with the empirical likelihood maximized over $\Theta$. The relevant statistic is
$$
\ell_{\Theta,n}(\boldsymbol{\mu}_0)
=-2\log
\frac{L_n(\boldsymbol{\mu}_0)}
{\sup_{\boldsymbol{\mu}\in\Theta}L_n(\boldsymbol{\mu})}.
$$
At a fixed interior point satisfying $\mu_{01}<\mu_{02}$, the unrestricted EL maximizer belongs to $\Theta$ with probability tending to one, and the constrained statistic has the ordinary $\chi^2_2$ limit. At a fixed boundary point $\boldsymbol{\mu}_0=(m,m)^\top$, however, both coordinates are specified under the null, while the denominator is maximized over a half-space. The resulting limit is
$$
\frac{1}{2}\chi^2_1+\frac{1}{2}\chi^2_2.
$$

This fixed-vector problem is different from testing the composite equality null
$$
H_0:\mu_1=\mu_2
\qquad\text{against}\qquad
H_1:\mu_1<\mu_2.
$$
For this test, the common mean under $H_0$ is unspecified and is profiled out. The corresponding statistic is
$$
T_n=
-2\log
\frac{\sup_{\mu_1=\mu_2}L_n(\boldsymbol{\mu})}
{\sup_{\boldsymbol{\mu}\in\Theta}L_n(\boldsymbol{\mu})},
$$
and its limiting distribution is
$$
\frac{1}{2}\chi^2_0+\frac{1}{2}\chi^2_1,
$$
where $\chi^2_0$ denotes a point mass at zero. This is the $k=2$ ordered-mean case of \citet[Theorem~3.2 and Example~2]{el1996empirical}. Although the two statistics concern the same boundary $\mu_1=\mu_2$, their null hypotheses have different dimensions. The tangential direction corresponding to the common mean contributes to the fixed-vector statistic but is removed by profiling in the composite test. Consequently, using the same critical value for the two problems would be incorrect.

The half-space structure yields an exact scalar profile. With $Z_i=X_{2i}-X_{1i}$, profiling the bivariate empirical likelihood subject to $\mu_2-\mu_1=d$ is equivalent to scalar empirical likelihood for $\mathrm{E}(Z_i)=d$, whenever the relevant empirical likelihood is feasible. This identity simplifies computation and supports localization of the constrained maximizer.

Because the fixed-interior and fixed-boundary limits provide only pointwise calibrations, we also consider a location-shift triangular array
$$
\boldsymbol{X}_{ni}=\boldsymbol{\mu}_n+\boldsymbol{\varepsilon}_{ni},
$$
in which the error law is fixed and $\boldsymbol{\mu}_n$ approaches the equality boundary at rate $n^{-1/2}$. Motivated by the broader boundary literature (\citet{andrews2001testing,balabdaoui2009likelihood}), we derive a local-to-boundary limit that connects the boundary chi-bar-square law with the interior $\chi_2^2$ law and shows that the corresponding fixed-point calibrations are not uniform near the boundary. Inversion over $\Theta$ produces an order-respecting joint region by construction. For $d=\mu_2-\mu_1$, however, the constrained scalar EL ratio must be distinguished from an ordinary scalar EL interval intersected with $[0,\infty)$: at $d=0$, their limits are $\tfrac12\chi_0^2+\tfrac12\chi_1^2$ and $\chi_1^2$, respectively.

The paper makes four specific contributions within this framework. First, it gives a self-contained half-space specialization and quadratic-projection derivation of the existing constrained-EL boundary theory, with localization established before the projection approximation is applied. Second, it uses the scalar profile identity to connect the bivariate constrained problem with scalar EL for paired differences and to distinguish fixed-vector inference from the known composite test of \citet{el1996empirical}. Third, it provides pointwise calibrated joint confidence regions, clarifies the two scalar interval constructions, and obtains an explicit EL limit for local sequences approaching the boundary. \citet{chen2011empirical} Fourth, it examines finite-sample coverage and directional power under Gaussian, Student $t_5$, and shifted log-normal sampling, covering negative, moderate positive, and strong positive dependence. Monte Carlo uncertainty, convex-hull infeasibility, and numerical convergence are reported explicitly. Public paired-data examples illustrate the different inferential targets and are treated as methodological illustrations rather than causal treatment-effect analyses.

The remainder of the paper is organized as follows. Section~2 defines the constrained empirical likelihood, establishes the scalar profile identity, and develops the fixed-candidate, composite-null, confidence-set, and local-to-boundary results. Section~3 presents the simulation study. Section~4 reports the illustrative data analyses, and Section~5 discusses the implications of the proposed calibrations. Proofs and detailed coverage and size results are reported in Appendix B, while complete machine-readable results are available in the repository.

\section{Empirical likelihood under the order restriction}

Let $\boldsymbol{X}_1,\ldots,\boldsymbol{X}_n$ be i.i.d.\ random vectors in $\R^2$, with
$$
  \E(\boldsymbol{X}_i)=\boldsymbol{\mu},\qquad
  \Sigma=\operatorname{Var}(\boldsymbol{X}_i),
$$
where $\Sigma$ is positive definite.  Put
$$
  \boldsymbol{a}=(-1,1)^\top,\qquad
  \Theta=\{\boldsymbol{\mu}:\boldsymbol{a}^\top\boldsymbol{\mu}\geq0\},\qquad
  H=\{\boldsymbol{\mu}:\boldsymbol{a}^\top\boldsymbol{\mu}=0\}.
$$
Thus $\boldsymbol{a}^\top\boldsymbol{\mu}=\mu_2-\mu_1$.

Let
$$
  \mathcal P_n
  =
  \left\{\boldsymbol{p}=(p_1,\ldots,p_n):p_i\geq0,\ 
  \sum_{i=1}^np_i=1\right\}.
$$
For a specified candidate mean $\boldsymbol{\mu}\in\R^2$, define the normalized empirical likelihood
$$
  L_n(\boldsymbol{\mu})
  =
  \sup\left\{
  \prod_{i=1}^n(np_i):
  \boldsymbol{p}\in\mathcal P_n,\ 
  \sum_{i=1}^np_i(\boldsymbol{X}_i-\boldsymbol{\mu})=\boldsymbol{0}
  \right\},
$$
where the supremum over an empty set is zero.  Hence $L_n(\bar{\boldsymbol{X}})=1$.  $L_n(\boldsymbol{\mu})$ is the largest multinomial likelihood over weights whose weighted sample mean equals $\boldsymbol{\mu}$.

Define the convex hull as
$$
\operatorname{conv}(\boldsymbol{X}_1,\ldots,\boldsymbol{X}_n)
=
\left\{
\sum_{i=1}^{n} p_i\boldsymbol{X}_i
:
p_i\geq 0,\quad
\sum_{i=1}^{n}p_i=1
\right\}.
$$
Now, if $\boldsymbol{\mu}\in\interior\conv(\boldsymbol{X}_1,\ldots,\boldsymbol{X}_n)$, the Lagrange-multiplier argument of \citet{owen1990,owen2001} gives
$$
  p_i(\boldsymbol{\mu})
  =
  \frac{1}{n\{1+\boldsymbol{\lambda}(\boldsymbol{\mu})^\top(\boldsymbol{X}_i-\boldsymbol{\mu})\}},
  \qquad i=1,\ldots,n,
$$
where $\boldsymbol{\lambda}(\boldsymbol{\mu})\in\R^2$ is the unique solution of
$$
  \sum_{i=1}^n
  \frac{\boldsymbol{X}_i-\boldsymbol{\mu}}{1+\boldsymbol{\lambda}(\boldsymbol{\mu})^\top(\boldsymbol{X}_i-\boldsymbol{\mu})}
  =\boldsymbol{0},
  \qquad
  1+\boldsymbol{\lambda}(\boldsymbol{\mu})^\top(\boldsymbol{X}_i-\boldsymbol{\mu})>0\quad\forall i.
$$
The ordinary empirical log-likelihood ratio is
$$
  \ell_n(\boldsymbol{\mu})
  =
  -2\log L_n(\boldsymbol{\mu})
  =
  2\sum_{i=1}^n
  \log\{1+\boldsymbol{\lambda}(\boldsymbol{\mu})^\top(\boldsymbol{X}_i-\boldsymbol{\mu})\}.
$$

For any nonempty set $A\subset\R^2$, let $\mathcal D_{A,n}=\{\sup_{\boldsymbol{\nu}\in A}L_n(\boldsymbol{\nu})>0\}$. For $\boldsymbol{\mu}\in\Theta$, on the event $\mathcal D_{\Theta,n}$ define the constrained ratio
\begin{equation}\label{eq:elltheta}
  \ell_{\Theta,n}(\boldsymbol{\mu})
  =
  -2\log
  \frac{L_n(\boldsymbol{\mu})}{\sup_{\boldsymbol{\nu}\in\Theta}L_n(\boldsymbol{\nu})}
  =
  \ell_n(\boldsymbol{\mu})-\inf_{\boldsymbol{\nu}\in\Theta}\ell_n(\boldsymbol{\nu}).
\end{equation}
$\ell_{\Theta,n}$ is not generally the ordinary ratio merely restricted to $\Theta$. On $\mathcal D_{\Theta,n}^{c}$ the denominator is zero, so the raw likelihood ratio is $0/0$ and is left undefined; such constrained-denominator infeasibility is recorded separately. For random-variable notation in asymptotic statements only, we use the arbitrary extension $\ell_{\Theta,n}=0$ on $\mathcal D_{\Theta,n}^{c}$; this extension does not affect any limit because the relevant feasibility probabilities tend to one.

\subsection{Exact profile reduction to paired differences}

Let
$$
  Z_i=\boldsymbol{a}^\top \boldsymbol{X}_i=X_{2i}-X_{1i},
$$
and $\bar Z=n^{-1}\sum_{i=1}^n Z_i$. Define the scalar empirical likelihood
$$
  L_{Z,n}(d)
  =
  \sup\left\{
  \prod_{i=1}^n(np_i):
  \boldsymbol{p}\in\mathcal P_n,\ 
  \sum_{i=1}^np_i(Z_i-d)=0
  \right\}.
$$
Write $\ell_{Z,n}(d)=-2\log L_{Z,n}(d)$, with $\ell_{Z,n}(d)=+\infty$ when $L_{Z,n}(d)=0$.

The following identity reduces the constrained bivariate optimization to scalar empirical likelihood for the paired differences, and is needed both for computation and for the later localization argument.

\begin{lemma}[Exact profiling]\label{lem:profile}
For every $d\in\R$,
$$
  \sup_{\boldsymbol{\mu}:\boldsymbol{a}^\top\boldsymbol{\mu}=d}L_n(\boldsymbol{\mu})=L_{Z,n}(d).
$$
Consequently,
$$
  \sup_{\boldsymbol{\mu}\in H}L_n(\boldsymbol{\mu})=L_{Z,n}(0),
  \qquad
  \sup_{\boldsymbol{\mu}\in\Theta}L_n(\boldsymbol{\mu})=\sup_{d\geq0}L_{Z,n}(d).
$$
\end{lemma}

\begin{proof}
For any $\boldsymbol{p}\in\mathcal P_n$, put $\boldsymbol{\mu}(\boldsymbol{p})=\sum_i p_i\boldsymbol{X}_i$.  Then
$$
  \boldsymbol{a}^\top\boldsymbol{\mu}(\boldsymbol{p})=\sum_i p_iZ_i.
$$
Thus maximizing over mean vectors with $\boldsymbol{a}^\top\boldsymbol{\mu}=d$ is exactly maximizing over weights with $\sum_i p_iZ_i=d$.  Taking $d=0$ or maximizing over $d\geq0$ gives the remaining claims.
\end{proof}

This finite-sample identity is the bivariate specialization underlying the ordered-mean example of \citet{el1996empirical}. It also gives a stable computational formula. Because $d\mapsto\log L_{Z,n}(d)$ is concave on its effective domain and is maximized at $\bar Z$, its maximum over $d\geq0$ occurs at $d=0$ when $\bar Z<0$ and $L_{Z,n}(0)>0$. If $\bar Z\geq0$, uniform weights are feasible for $\Theta$, so $\sup_{\Theta}L_n=1$. If $\bar Z<0$ and $L_{Z,n}(0)>0$, the order constraint binds and
$$
  -2\log\sup_{\boldsymbol{\mu}\in\Theta}L_n(\boldsymbol{\mu})=\ell_{Z,n}(0).
$$
Therefore, on the event $\{\bar Z\geq0\}\cup\{L_{Z,n}(0)>0\}$,
$$
  \ell_{\Theta,n}(\boldsymbol{\mu})
  =
  \ell_n(\boldsymbol{\mu})
  -
  \mathbf 1\{\bar Z<0\}\ell_{Z,n}(0).
$$
Under the boundary assumptions below, this event has probability tending to one. When $\bar Z<0$, the constrained denominator is positive if and only if $\min_i Z_i<0<\max_i Z_i$; if $\max_i Z_i\leq0$ and at least one $Z_i<0$, it is zero, so the likelihood ratio is undefined and this convex-hull infeasibility must be recorded separately.

\subsection{Fixed interior point}

We first record the interior benchmark, because it identifies the regime in which the order restriction is asymptotically inactive and ordinary Wilks calibration remains valid.

\begin{theorem}[Fixed interior mean vector]\label{thm:interior}
Suppose $\E(\boldsymbol{X}_i)=\boldsymbol{\mu}_0$, $\E\|\boldsymbol{X}_i\|^2<\infty$, $\Sigma$ is positive definite, and $\boldsymbol{\mu}_0\in\interior(\Theta)$.  Then
$$
  \ell_{\Theta,n}(\boldsymbol{\mu}_0)\ \xrightarrow{d}\ \chi^2_2.
$$
\end{theorem}

\begin{proof}
If $\boldsymbol{a}^\top\boldsymbol{\mu}_0>0$, the law of large numbers gives $\Pp(\bar{\boldsymbol{X}}\in\Theta)\to1$.  On this event, the unconstrained EL maximizer $\bar{\boldsymbol{X}}$ is feasible, and $\ell_{\Theta,n}(\boldsymbol{\mu}_0)=\ell_n(\boldsymbol{\mu}_0)$.  The conclusion follows from the ordinary multivariate EL Wilks theorem \citet{owen1990,owen2001}.
\end{proof}

\subsection{Fixed boundary point}\label{sec:boundary}

Let $\boldsymbol{\mu}_0=(m,m)^\top\in H$ and define
$$
  \boldsymbol{\Delta}_n=\sqrt n(\bar{\boldsymbol{X}}-\boldsymbol{\mu}_0),
  \qquad
  K=\{\boldsymbol{h}\in\R^2:\boldsymbol{a}^\top \boldsymbol{h}\geq0\}.
$$
Thus, the statement that the restriction is asymptotically active has the following precise meaning:
$$
  \sqrt n\,\boldsymbol{a}^\top(\bar{\boldsymbol{X}}-\boldsymbol{\mu}_0)
  \ \xrightarrow{d}\
  N(0,\boldsymbol{a}^\top\Sigma\boldsymbol{a}),
  \qquad
  \Pp(\bar{\boldsymbol{X}}\notin\Theta)\to\tfrac12.
$$
Thus the unconstrained EL maximizer is infeasible with a nonvanishing limiting probability, and the constrained maximizer then lies on $H$.

We use the following assumptions.
\begin{enumerate}[label=(B\arabic*),leftmargin=*,itemsep=2pt]
\item $\boldsymbol{X}_1,\ldots,\boldsymbol{X}_n$ are i.i.d.\ with common distribution $F$ and mean $\boldsymbol{\mu}_0$.
\item $\Sigma=\operatorname{Var}(\boldsymbol{X}_i)$ is positive definite.
\item $\E\|\boldsymbol{X}_i-\boldsymbol{\mu}_0\|^3<\infty$.
\item For every fixed $M<\infty$,
$$
    \Pp\left[
    \boldsymbol{\mu}_0+n^{-1/2}\{\boldsymbol{h}:\|\boldsymbol{h}\|\leq M\}
    \subset\interior\conv(\boldsymbol{X}_1,\ldots,\boldsymbol{X}_n)
    \right]\to1.
$$
\end{enumerate}
Assumption (B4) is the local convex-hull feasibility condition needed to ensure that the empirical-likelihood multiplier exists uniformly on root-$n$ neighborhoods. In the present finite-dimensional mean problem, it follows from (B1)--(B2): if a supporting hyperplane through $\boldsymbol{\mu}_0$ existed, then some nonzero $\boldsymbol{v}$ would make $\boldsymbol{v}^{\top}(\boldsymbol{X}_i-\boldsymbol{\mu}_0)$ one-signed with mean zero, hence zero almost surely, contradicting $\boldsymbol{v}^{\top}\Sigma\boldsymbol{v}>0$; thus $\boldsymbol{\mu}_0\in\interior\conv(\operatorname{supp}F)$, and finitely many support neighborhoods whose convex hull contains a neighborhood of $\boldsymbol{\mu}_0$ are all hit by the sample with probability tending to one. We nevertheless state (B4) explicitly to make the geometric feasibility step used in the uniform expansion transparent (\citet{owen1990, owen2001}). \\


\noindent To prove the main theorem, we first establish the following preliminary lemmas.
\begin{lemma}[Uniform local quadratic expansion]\label{lem:quadratic}
Under (B1)--(B4), for every fixed $M<\infty$,
$$
  \sup_{\|\boldsymbol{h}\|\leq M}
  \left|
  \ell_n\left(\boldsymbol{\mu}_0+\frac{\boldsymbol{h}}{\sqrt n}\right)
  -
  (\boldsymbol{\Delta}_n-\boldsymbol{h})^\top\Sigma^{-1}(\boldsymbol{\Delta}_n-\boldsymbol{h})
  \right|
  \xrightarrow{\Pp}0.
$$
\end{lemma}

Because the preceding expansion is only local, the next result is required to show that the constrained and equality-profile maximizers lie in the root-$n$ neighborhood on which that expansion is valid.

\begin{lemma}[Localization]\label{lem:localization}
Under (B1)--(B4), select the following maximizers on $\mathcal D_{\Theta,n}$ and $\mathcal D_{H,n}$, respectively, and assign them the arbitrary value $\boldsymbol{\mu}_0$ on the corresponding complements:
$$
  \widehat{\boldsymbol{\mu}}_\Theta\in\argmax_{\boldsymbol{\mu}\in\Theta}L_n(\boldsymbol{\mu}),
  \qquad
  \widehat{\boldsymbol{\mu}}_H\in\argmax_{\boldsymbol{\mu}\in H}L_n(\boldsymbol{\mu}).
$$
Then
$$
  \widehat{\boldsymbol{\mu}}_\Theta-\boldsymbol{\mu}_0=O_{p}(n^{-1/2}),
  \qquad
  \widehat{\boldsymbol{\mu}}_H-\boldsymbol{\mu}_0=O_{p}(n^{-1/2}).
$$
\end{lemma}

The proofs are given in Appendix A. 

Combining uniform expansion with localization converts the constrained empirical-likelihood ratio into a metric-projection problem, which is the form needed to identify its boundary limit.

\begin{lemma}[Quadratic projection reduction]\label{lem:reduction}
Under (B1)--(B4),
\begin{equation}\label{eq:projection-reduction}
  \ell_{\Theta,n}(\boldsymbol{\mu}_0)
  =
  \boldsymbol{\Delta}_n^\top\Sigma^{-1}\boldsymbol{\Delta}_n
  -
  \inf_{\boldsymbol{h}\in K}
  (\boldsymbol{\Delta}_n-\boldsymbol{h})^\top\Sigma^{-1}(\boldsymbol{\Delta}_n-\boldsymbol{h})
  +o_{p}(1).
\end{equation}
\end{lemma}

The projection representation now yields the correct calibration for a completely specified mean vector on the equality boundary, where the order constraint is active with nonvanishing probability (\citet{el1996empirical}). 

\begin{theorem}[Fixed boundary point]\label{thm:fixed-boundary}
Under (B1)--(B4),
$$
  \ell_{\Theta,n}(\boldsymbol{\mu}_0)
  \ \xrightarrow{d}\
  \barchi{1}{2},
$$
where $\barchi{1}{2}$ denotes the mixture distribution with cdf
$$
  \Pp(\barchi{1}{2}\leq x)
  =
  \tfrac12\Pp(\chi^2_1\leq x)
  +
  \tfrac12\Pp(\chi^2_2\leq x),
  \qquad x\geq0.
$$
\end{theorem}

\begin{proof}

Define
$$
Q(\boldsymbol{g})
=
\boldsymbol{g}^{\top}\boldsymbol{\Sigma}^{-1}\boldsymbol{g}
-
\inf_{\boldsymbol{h}\in K}
(\boldsymbol{g}-\boldsymbol{h})^{\top}
\boldsymbol{\Sigma}^{-1}
(\boldsymbol{g}-\boldsymbol{h}),
$$
where
$$
K
=
\left\{
\boldsymbol{h}\in\mathbb{R}^{2}:
\boldsymbol{a}^{\top}\boldsymbol{h}\geq 0
\right\}.
$$
Because $K$ is closed, the distance from $\boldsymbol{g}$ to $K$ in the norm induced by $\boldsymbol{\Sigma}^{-1}$ is continuous in $\boldsymbol{g}$. Therefore, by Lemma~\ref{lem:reduction}, the central limit theorem, the continuous mapping theorem, and Slutsky's theorem,
$$
\ell_{\Theta,n}(\boldsymbol{\mu}_0)
\ \xrightarrow{d}\
Q(\boldsymbol{G}),
\qquad
\boldsymbol{G}\sim
N_2(\boldsymbol{0},\boldsymbol{\Sigma}).
$$

We first determine the $\boldsymbol{\Sigma}^{-1}$-metric projection of $\boldsymbol{G}$ onto $K$. We distinguish two cases.

\medskip
\noindent\emph{Case 1: $\boldsymbol{a}^{\top}\boldsymbol{G}\geq 0$.}
Then $\boldsymbol{G}\in K$, so the infimum is attained at $\boldsymbol{h}=\boldsymbol{G}$. Hence
$$
Q(\boldsymbol{G})
=
\boldsymbol{G}^{\top}
\boldsymbol{\Sigma}^{-1}
\boldsymbol{G}.
$$

\medskip
\noindent\emph{Case 2: $\boldsymbol{a}^{\top}\boldsymbol{G}<0$.}
Then $\boldsymbol{G}\notin K$, and its $\boldsymbol{\Sigma}^{-1}$-metric projection onto $K$ lies on the boundary
$$
\left\{
\boldsymbol{h}\in\mathbb{R}^{2}:
\boldsymbol{a}^{\top}\boldsymbol{h}=0
\right\}.
$$
Minimizing
$$
(\boldsymbol{G}-\boldsymbol{h})^{\top}
\boldsymbol{\Sigma}^{-1}
(\boldsymbol{G}-\boldsymbol{h})
$$
subject to $\boldsymbol{a}^{\top}\boldsymbol{h}=0$ gives
$$
\boldsymbol{h}^{*}
=
\boldsymbol{G}
-
\frac{\boldsymbol{a}^{\top}\boldsymbol{G}}
     {\boldsymbol{a}^{\top}\boldsymbol{\Sigma}\boldsymbol{a}}
\boldsymbol{\Sigma}\boldsymbol{a}.
$$
Indeed, $\boldsymbol{a}^{\top}\boldsymbol{h}^{*}=0$, and therefore
$$
\inf_{\boldsymbol{h}\in K}
(\boldsymbol{G}-\boldsymbol{h})^{\top}
\boldsymbol{\Sigma}^{-1}
(\boldsymbol{G}-\boldsymbol{h})
=
\frac{
(\boldsymbol{a}^{\top}\boldsymbol{G})^{2}
}{
\boldsymbol{a}^{\top}\boldsymbol{\Sigma}\boldsymbol{a}
}.
$$
Consequently,
$$
Q(\boldsymbol{G})
=
\boldsymbol{G}^{\top}
\boldsymbol{\Sigma}^{-1}
\boldsymbol{G}
-
\frac{
(\boldsymbol{a}^{\top}\boldsymbol{G})^{2}
}{
\boldsymbol{a}^{\top}\boldsymbol{\Sigma}\boldsymbol{a}
}.
$$

Combining the two cases,
\begin{equation}
Q(\boldsymbol{G})
=
\begin{cases}
\boldsymbol{G}^{\top}
\boldsymbol{\Sigma}^{-1}
\boldsymbol{G},
&
\boldsymbol{a}^{\top}\boldsymbol{G}\geq 0,
\\[1.2ex]
\displaystyle
\boldsymbol{G}^{\top}
\boldsymbol{\Sigma}^{-1}
\boldsymbol{G}
-
\frac{
(\boldsymbol{a}^{\top}\boldsymbol{G})^{2}
}{
\boldsymbol{a}^{\top}\boldsymbol{\Sigma}\boldsymbol{a}
},
&
\boldsymbol{a}^{\top}\boldsymbol{G}<0.
\end{cases}
\label{eq:boundary_piecewise_Q}
\end{equation}

Let $\boldsymbol{\Sigma}^{1/2}$ denote the symmetric positive-definite square root of $\boldsymbol{\Sigma}$, and define
$$
\boldsymbol{Y}
=
\boldsymbol{\Sigma}^{-1/2}\boldsymbol{G}
\sim N_2(\boldsymbol{0},I_2),
\qquad
\boldsymbol{c}
=
\boldsymbol{\Sigma}^{1/2}\boldsymbol{a}.
$$
Choose an orthonormal basis
$\{\boldsymbol{e}_1,\boldsymbol{e}_2\}$ such that
$$
\boldsymbol{e}_1
=
\frac{\boldsymbol{c}}{\|\boldsymbol{c}\|},
$$
and write
$$
\boldsymbol{Y}
=
U\boldsymbol{e}_1
+
V\boldsymbol{e}_2,
$$
where $U$ and $V$ are independent standard normal random variables. Then
$$
\boldsymbol{a}^{\top}\boldsymbol{G}
=
\boldsymbol{c}^{\top}\boldsymbol{Y}
=
\|\boldsymbol{c}\|U,
$$
and
$$
\boldsymbol{G}^{\top}
\boldsymbol{\Sigma}^{-1}
\boldsymbol{G}
=
\boldsymbol{Y}^{\top}\boldsymbol{Y}
=
U^2+V^2.
$$
Moreover, because
$$
\boldsymbol{a}^{\top}\boldsymbol{\Sigma}\boldsymbol{a}
=
\|\boldsymbol{c}\|^2,
$$
we have
$$
\frac{
(\boldsymbol{a}^{\top}\boldsymbol{G})^2
}{
\boldsymbol{a}^{\top}\boldsymbol{\Sigma}\boldsymbol{a}
}
=
U^2.
$$
Substituting these identities into
\eqref{eq:boundary_piecewise_Q} gives
$$
Q(\boldsymbol{G})
=
\begin{cases}
U^2+V^2,
&
U\geq 0,
\\
V^2,
&
U<0.
\end{cases}
$$
Therefore, for every $x\geq 0$,
\begin{align*}
\Pp\{Q(\boldsymbol{G})\leq x\}
&=
\Pp(U^2+V^2\leq x,\ U\geq 0)
+
\Pp(V^2\leq x,\ U<0)
\\
&=
\frac{1}{2}\Pp(U^2+V^2\leq x)
+
\frac{1}{2}\Pp(V^2\leq x)
\\
&=
\frac{1}{2}\Pp(\chi_2^2\leq x)
+
\frac{1}{2}\Pp(\chi_1^2\leq x).
\end{align*}
The second equality follows because the sign of $U$ is independent of $U^2$, and, since $U$ and $V$ are independent, it is also independent of the pair $(U^2,V^2)$. Moreover, each sign has probability $1/2$. Thus
$$
Q(\boldsymbol{G})
\sim
\frac{1}{2}\chi_1^2
+
\frac{1}{2}\chi_2^2
=
\barchi{1}{2},
$$
which proves the result.

\end{proof}

The 0.95 quantile of $\barchi{1}{2}$ is 5.138, compared with $\chi^2_{2,0.95}=5.991$.  Thus $\chi^2_2$ calibration is asymptotically conservative for this fixed boundary-point statistic.

\subsection{Composite equality null versus ordered alternative}

On the event $\mathcal D_{H,n}$, define
$$
  T_n
  =
  -2\log
  \frac{\sup_{\boldsymbol{\mu}\in H}L_n(\boldsymbol{\mu})}
       {\sup_{\boldsymbol{\mu}\in\Theta}L_n(\boldsymbol{\mu})}.
$$
On $\mathcal D_{H,n}^{c}$ the raw statistic is undefined and the infeasibility is recorded; for the asymptotic statement below only, set $T_n=0$ on this event.
This is not the statistic in Theorem~\ref{thm:fixed-boundary}: the common mean is profiled out under $H$.

The next theorem gives the calibration for this distinct composite null and is needed to avoid applying the fixed-vector boundary critical value after the common mean has been profiled out.

\begin{theorem}[El Barmi's $k=2$ equality test]\label{thm:composite}
Suppose $\boldsymbol{X}_1,\ldots,\boldsymbol{X}_n$ are i.i.d., the true mean belongs to $H$, $\E|Z_i|^3<\infty$, and $\sigma_Z^2=\operatorname{Var}(Z_i)=\boldsymbol{a}^\top\Sigma\boldsymbol{a}>0$. Then
$$
  T_n\ \xrightarrow{d}\ \barchi{0}{1}
  =
  \tfrac12\chi^2_0+\tfrac12\chi^2_1,
$$
where $\chi^2_0$ is a point mass at zero.
\end{theorem}

\begin{proof}
On the event that the sample contains both positive and negative $Z_i$'s, Lemma~\ref{lem:profile} gives the exact identity
$$
  T_n
  =
  \mathbf 1\{\bar Z\geq0\}\ell_{Z,n}(0).
$$
The complement of this event has probability tending to zero. Standard scalar EL theory yields the following expansion (\citet{owen1988empirical}). 
$$
  \ell_{Z,n}(0)
  =
  \frac{n\bar Z^2}{\sigma_Z^2}+o_{p}(1).
$$
Since $\sqrt n\,\bar Z/\sigma_Z\Rightarrow U\sim N(0,1)$,
$$
  T_n\Rightarrow (U_+)^2,\qquad U_+=\max(U,0),
$$
which has the stated distribution.
\end{proof}

Theorem~\ref{thm:composite} is the $k=2$ specialization of \citet[Theorem~3.2 and Example~2]{el1996empirical}. Let $\Phi$ denote the standard normal cdf and let $z_q=\Phi^{-1}(q)$. For $0<\alpha<1/2$, its positive $1-\alpha$ critical value is
$$
  \chi^2_{1,1-2\alpha}=z_{1-\alpha}^2.
$$
At level 0.05 this equals 2.706.  For an observed $T_n>0$, the asymptotic directional $p$-value is
$$
  \tfrac12\Pp(\chi^2_1\geq T_n)
  =
  1-\Phi(\sqrt{T_n}).
$$
When $\bar Z<0$, the likelihood-ratio statistic collapses to zero. For descriptive reporting we also use the signed root
$$
  r_n=\operatorname{sign}(\bar Z)\sqrt{\ell_{Z,n}(0)},
  \qquad p_{\mathrm{sr}}=1-\Phi(r_n),
$$
which retains the direction discarded by $T_n$ and gives the same rejection rule for positive evidence.

\subsection{Confidence sets and scalar intervals}

Let $c_{12,1-\alpha}$ be the $1-\alpha$ quantile of $\barchi{1}{2}$.  A candidate-wise calibrated joint confidence region is
\begin{align*}
  C^{\mathrm{pt}}_{1-\alpha}
  ={}&
  \{\boldsymbol{\mu}\in\interior(\Theta):
  \ell_{\Theta,n}(\boldsymbol{\mu})\leq\chi^2_{2,1-\alpha}\}\\
  &{}\cup
  \{\boldsymbol{\mu}\in H:
  \ell_{\Theta,n}(\boldsymbol{\mu})\leq c_{12,1-\alpha}\}.
\end{align*}
Under the assumptions of Theorem~\ref{thm:interior} at an interior point and (B1)--(B4) at a boundary point, for every fixed true distribution $F$ with mean $\boldsymbol{\mu}_0\in\Theta$,
$$
  \Pp_F\{\boldsymbol{\mu}_0\in C^{\mathrm{pt}}_{1-\alpha}\}\to1-\alpha.
$$
This statement is pointwise, not uniform in a shrinking neighborhood of $H$.  A closed, uniformly defined alternative is
$$
  C^{\chi^2}_{1-\alpha}
  =
  \{\boldsymbol{\mu}\in\Theta:
  \ell_{\Theta,n}(\boldsymbol{\mu})\leq\chi^2_{2,1-\alpha}\}.
$$
It is first-order exact at fixed interior points and conservative at a fixed boundary point.

For the scalar difference $d=\mu_2-\mu_1$, two valid constructions must not be conflated.  The constrained ratio is
$$
  \ell_{Z,+,n}(d)
  =
  -2\log
  \frac{L_{Z,n}(d)}{\sup_{u\geq0}L_{Z,n}(u)},
  \qquad d\geq0.
$$
Under the scalar EL conditions in Theorem~\ref{thm:composite}, it has a $\chi^2_1$ limit at $d_0>0$ and a $\barchi{0}{1}$ limit at $d_0=0$.  In contrast, the interval used in our data illustrations is the ordinary scalar EL interval intersected with the scientific parameter space:
\begin{equation}\label{eq:truncated-interval}
  I^{\mathrm{tr}}_{1-\alpha}
  =
  \{d\geq0:\ell_{Z,n}(d)\leq\chi^2_{1,1-\alpha}\}.
\end{equation}
Because the denominator in $\ell_{Z,n}$ is unrestricted over $d\in\R$, the ordinary $\chi^2_1$ limit applies at $d_0=0$ as well as at $d_0>0$ (\citet{owen1988empirical}).  Accordingly, \eqref{eq:truncated-interval} is not described as inversion of the constrained ratio. This resolves the apparent conflict between ordinary scalar and chi-bar boundary calibration.

\subsection{Local sequences approaching the boundary}

To assess whether the fixed-candidate calibrations remain valid near $H$, the next result derives the limit along $n^{-1/2}$-local sequences and thereby quantifies the nonuniform transition between the boundary and interior regimes.

\begin{proposition}[Local-to-boundary law]\label{prop:local}
For each $n$, let $\boldsymbol{X}_{ni}=\boldsymbol{\mu}_n+\boldsymbol{\varepsilon}_{ni}$, $i=1,\ldots,n$, where the vectors within each row are i.i.d. and the common distribution of $\boldsymbol{\varepsilon}_{ni}$ does not depend on $n$, with $\E(\boldsymbol{\varepsilon}_{ni})=\boldsymbol{0}$, positive-definite covariance matrix $\Sigma$, and $\E\|\boldsymbol{\varepsilon}_{ni}\|^3<\infty$. Assume that, for every fixed $M<\infty$,
$$
  \Pp\left[
  \boldsymbol{\mu}_n+n^{-1/2}\{\boldsymbol{h}:\|\boldsymbol{h}\|\leq M\}
  \subset\interior\conv(\boldsymbol{X}_{n1},\ldots,\boldsymbol{X}_{nn})
  \right]\to1.
$$
If $\boldsymbol{\mu}_n\in\Theta$, $\boldsymbol{\mu}_n\to\boldsymbol{\mu}_\star\in H$, and
$$
  \frac{\sqrt n\,\boldsymbol{a}^\top\boldsymbol{\mu}_n}
       {\sqrt{\boldsymbol{a}^\top\Sigma\boldsymbol{a}}}
  \longrightarrow \eta\in[0,\infty),
$$
then
$$
  \ell_{\Theta,n}(\boldsymbol{\mu}_n)\Rightarrow Q_\eta,
$$
where, for independent $U,V\sim N(0,1)$,
$$
  Q_\eta
  =
  \begin{cases}
    U^2+V^2,&U\geq-\eta,\\[2mm]
    U^2+V^2-(U+\eta)^2,&U<-\eta.
  \end{cases}
$$
At $\eta=0$, $Q_0\sim\barchi{1}{2}$; as $\eta\to\infty$, $Q_\eta$ approaches $\chi^2_2$.
\end{proposition}

For every finite $\eta>0$, the coupling in Proposition~\ref{prop:local} gives $Q_\eta\leq U^2+V^2$ almost surely. Thus using the interior $\chi^2_2$ critical value along such a local sequence has limiting coverage at least $1-\alpha$, generally strictly greater than $1-\alpha$, and approaches exact coverage only as $\eta\to\infty$. Any finite-sample undercoverage in the simulations below is therefore a higher-order error rather than a contradiction of the local limit.

For the composite test under
$$
  d_n=\eta\sigma_Z/\sqrt n,
$$
the same scalar expansion gives
$$
  T_n\Rightarrow(U+\eta)_+^2.
$$
Therefore, for $0<\alpha<1/2$, the local asymptotic power at level $\alpha$ is
\begin{equation}\label{eq:local-power}
  1-\Phi(z_{1-\alpha}-\eta).
\end{equation}

\section{Simulation study}\label{sec:simulation}

We conducted Monte Carlo experiments to assess the finite-sample behavior of the empirical likelihood procedures in the interior of the order cone, on its boundary, under local sequences approaching the boundary, and under the composite equality null. These settings correspond to the distinct asymptotic regimes developed in Section~2. At a fixed interior point, the constrained empirical likelihood ratio has the usual Wilks limit $\chi^2_2$.

For each simulated sample in the fixed-vector experiments, inference was based on $\ell_{\Theta,n}(\boldsymbol{\mu}_0)$. At the nominal $5\%$ level, the interior experiments used the $\chi^2_2$ critical value $5.991$. The fixed-boundary experiments compared the appropriate $\barchi{1}{2}$ critical value $5.138$ with the ordinary $\chi^2_2$ critical value. In the local-to-boundary experiments, the boundary critical value was used at $\eta=0$, whereas the interior critical value was used at $\eta>0$ to examine the nonuniformity of these pointwise calibrations. The composite equality test was based on $T_n$ and the analytic $\barchi{0}{1}$ critical value $2.706$; no bootstrap calibration was used.

We report empirical coverage probabilities or rejection probabilities, as appropriate, together with their Monte Carlo standard errors. We also report the boundary-optimization rate, defined as the proportion of replications for which the sample mean violated the order restriction and the constrained empirical likelihood denominator was therefore maximized over the boundary $\mu_1=\mu_2$. At a fixed interior point, this rate should converge to zero, whereas under the fixed boundary null it should converge to one half, reflecting the half-space geometry underlying the chi-bar-square limit.

\subsection*{Data-generating models}\label{subsec:data-generating-models}

In all simulations, $\boldsymbol{X}_i=(X_{1i},X_{2i})^\top$, $i=1,\ldots,n$, were generated independently with target mean vector $\boldsymbol{\mu}=(\mu_1,\mu_2)^\top$. Define

$$
\boldsymbol{\Sigma}(\rho)
=
\begin{pmatrix}
1 & \rho\\
\rho & 1
\end{pmatrix},
\qquad
\rho\in\{-0.5,0.3,0.8\}.
$$

These values represent negative, moderate positive, and strong positive dependence, respectively. In the Gaussian and Student $t_5$ models, $\rho$ is the Pearson correlation between the observed components. In the shifted log-normal model, $\rho$ is the correlation between the latent Gaussian components.

\paragraph{Model 1: Bivariate normal} 
Under the Gaussian design, $
\boldsymbol{X}_i
\sim
N_2\!\left\{
\boldsymbol{\mu},
\boldsymbol{\Sigma}(\rho)
\right\}.
$

\paragraph{Model 2: Standardized bivariate Student $t_5$.}
For the heavy-tailed design, we generated

$$
\boldsymbol{G}_i
\sim
N_2\!\left\{
\boldsymbol{0},
\boldsymbol{\Sigma}(\rho)
\right\},
\qquad
W_i\sim\chi^2_5,
$$

where $\boldsymbol{G}_i$ and $W_i$ were independent, and defined

$$
\boldsymbol{X}_i
=
\boldsymbol{\mu}
+
\sqrt{\frac{3}{5}}\,
\frac{\boldsymbol{G}_i}{\sqrt{W_i/5}}.
$$

The factor $\sqrt{3/5}$ standardizes the distribution so that
$
\E(\boldsymbol{X}_i)=\boldsymbol{\mu},
\;
\operatorname{Cov}(\boldsymbol{X}_i)
=
\boldsymbol{\Sigma}(\rho).
$
This model preserves the covariance structure of the Gaussian design while introducing heavier tails.

\paragraph{Model 3: Shifted correlated log-normal.}
For the skewed design, we first generated

$$
\boldsymbol{G}_i
=
(G_{1i},G_{2i})^\top
\sim
N_2\!\left\{
\boldsymbol{0},
\boldsymbol{\Sigma}(\rho)
\right\}.
$$

With $\sigma_L=0.6$, we then defined

$$
X_{ji}
=
\mu_j
+
\frac{
\exp(\sigma_LG_{ji})-\exp(\sigma_L^2/2)
}{
\left\{
\exp(2\sigma_L^2)-\exp(\sigma_L^2)
\right\}^{1/2}
},
\qquad
j=1,2.
$$

The centering and scaling ensure that $
\E(X_{ji})=\mu_j,
\;
\operatorname{Var}(X_{ji})=1.
$ For this model, $\rho$ is the latent Gaussian correlation and is not the Pearson correlation between $X_{1i}$ and $X_{2i}$. The observed correlation is

$$
\operatorname{Corr}(X_{1i},X_{2i})
=
\frac{
\exp(\sigma_L^2\rho)-1
}{
\exp(\sigma_L^2)-1
}.
$$

For $\rho=-0.5$, $0.3$, and $0.8$, the corresponding observed correlations are approximately $-0.380$, $0.263$, and $0.770$, respectively. Unlike the Gaussian and Student $t_5$ designs, this distribution is not elliptically symmetric and has positively skewed marginal errors.

We used $n\in\{30,60,100\}$. Fixed-point coverage and size used 5,000 replications per setting; power used 3,000. The fixed interior mean was $(0,1)^\top$, and the fixed boundary mean was $(0,0)^\top$. Local means were
$$
  \boldsymbol{\mu}_n=(0,d_n)^\top,\qquad
  d_n=\eta\{\operatorname{Var}(X_2-X_1)/n\}^{1/2},
  \quad
  \eta\in\{0,0.5,1,2,4\}.
$$
Standardizing by the standard deviation of the paired difference avoids confounding dependence with effect size.

For an estimated probability $\widehat p$ based on $R$ replications, we report
$$
  \operatorname{MCSE}(\widehat p)
  =
  \{\widehat p(1-\widehat p)/R\}^{1/2}.
$$
The 95\% Monte Carlo intervals shown in the figures are the normal-approximation intervals $\widehat p\pm1.96\operatorname{MCSE}(\widehat p)$.

\subsection{Fixed interior and fixed boundary results}

We first examined the fixed interior point $\boldsymbol{\mu}_0=(0,1)^\top$, for which $\mu{01}<\mu_{02}$. Because the true mean is separated from the boundary, the order restriction is asymptotically inactive, and  $\ell_{\Theta,n}(\boldsymbol{\mu}_0)\overset{d}\longrightarrow\chi^2_2$. This experiment evaluates the finite-sample accuracy of the ordinary Wilks calibration. We next considered the fixed boundary point $\boldsymbol{\mu}_0=(0,0)^\top$, where $\mu_{01}=\mu_{02}$. In this case, Theorem~\ref{thm:fixed-boundary} gives the nonstandard limit $\barchi{1}{2}=\tfrac12\chi^2_1+\tfrac12\chi^2_2$. The simulation therefore compares the theoretically correct chi-bar-square calibration with the ordinary $\chi^2_2$ calibration, which is conservative to first order at the boundary, although finite-sample tail error may partially offset this conservatism.

\begin{table}[!ht]
\centering
\caption{Ranges over the reported $\rho$ and $n$ settings and, for local coverage, over the reported $\eta$ settings. Local coverage uses the boundary critical value at $\eta=0$ and the interior critical value for $\eta>0$.}
\label{tab:summary}
\small
\begin{tabular}{lcccc}
\toprule
Model & Interior coverage & Fixed-boundary size &
Composite size & Local coverage\\
\midrule
Normal & 0.924--0.949 & 0.053--0.073 & 0.045--0.055 & 0.926--0.964\\
Student $t_5$ & 0.912--0.943 & 0.060--0.086 & 0.051--0.068 & 0.909--0.957\\
Log-normal & 0.893--0.936 & 0.057--0.108 & 0.053--0.070 & 0.888--0.949\\
\bottomrule
\end{tabular}
\end{table}

Table~\ref{tab:summary} summarizes the finite-sample performance under the fixed interior and fixed boundary regimes. In the interior of the order cone, coverage based on the asymptotic $\chi^2_2$ calibration was generally below the nominal level, ranging from $0.924$ to $0.949$ under normal sampling, from $0.912$ to $0.943$ under Student $t_5$ sampling, and from $0.893$ to $0.936$ under shifted log-normal sampling. Consistent with this undercoverage, the empirical $0.95$ quantiles of the constrained empirical likelihood ratio exceeded the asymptotic critical value $\chi^2_{2,0.95}=5.991$ in every reported setting; see Appendix~B. The discrepancy decreased as $n$ increased but remained visible at $n=100$, particularly for the heavy-tailed and skewed distributions. Thus, although the ordinary Wilks calibration is asymptotically valid at a fixed interior point, its finite-sample accuracy can be inadequate when the underlying distribution is markedly non-Gaussian.

At the fixed boundary point $\boldsymbol{\mu}_0=(0,0)^\top$, calibration by the theoretically correct chi-bar-square distribution $\barchi{1}{2}=\tfrac12\chi^2_1+\tfrac12\chi^2_2$ also produced some upward size distortion. The empirical rejection probabilities ranged from $0.053$ to $0.073$ under normal sampling, from $0.060$ to $0.086$ under Student $t_5$ sampling, and from $0.057$ to $0.108$ under shifted log-normal sampling. The distortion was most pronounced at $n=30$ and under heavy-tailed or skewed sampling, and it generally decreased with increasing sample size. These results indicate that the chi-bar-square limit correctly describes the boundary asymptotic regime, but first-order calibration alone does not eliminate finite-sample error.

The ordinary $\chi^2_2$ critical value occasionally yielded rejection probabilities numerically closer to $0.05$, particularly in some of the smaller-sample settings. This apparent improvement results from an accidental compensation: the larger $\chi^2_2$ critical value offsets the upward finite-sample tail error of the empirical likelihood statistic. It should not be interpreted as evidence in favor of $\chi^2_2$ calibration at the boundary. Indeed, as the finite-sample error diminishes, the $\chi^2_2$ calibration becomes conservative because it does not account for the boundary geometry.

Finally, the proportion of replications for which $\bar Z<0$, and hence the constrained denominator maximizer lay on the boundary set $H$, ranged from $0.493$ to $0.515$. This is close to the limiting probability $1/2$ and provides empirical support for the equal mixing weights in $\barchi{1}{2}$. Overall, the simulations support the distinction between the interior and boundary limiting distributions while also showing that asymptotic correctness does not necessarily imply accurate inference at the moderate sample sizes considered here. Complete setting-specific results and Monte Carlo standard errors are reported in Appendix~B.

\subsection{Local-to-boundary coverage}
To examine the transition between the boundary and interior regimes, we considered the local sequence $d_n=\eta\,\operatorname{sd}(X_2-X_1)/\sqrt n$, with $\eta\in\{0,0.5,1,2,4\}$. Proposition~\ref{prop:local} shows that the limiting distribution is $Q_\eta$, which equals $\barchi{1}{2}$ at $\eta=0$ and approaches $\chi^2_2$ as $\eta\to\infty$. This experiment evaluates the nonuniform behavior of the pointwise boundary/interior calibration within an $n^{-1/2}$-neighborhood of the boundary.

\begin{figure}[!ht]
\centering
\includegraphics[width=\textwidth]{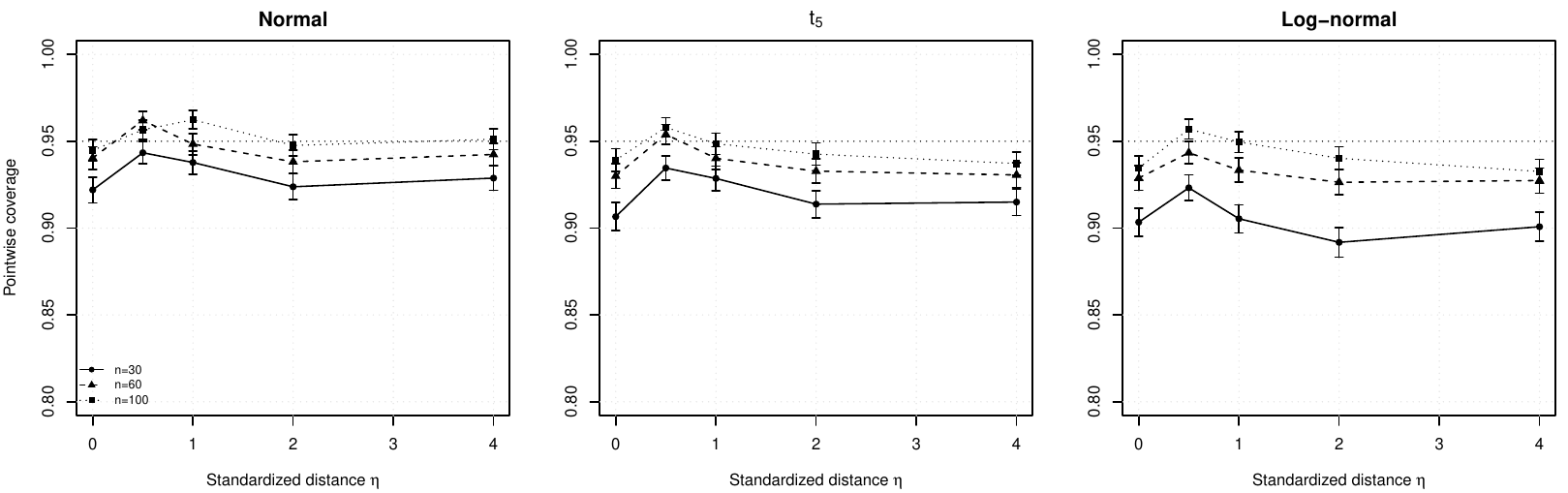}
\caption{Pointwise coverage along $d_n=\eta\,\operatorname{sd}(X_2-X_1)/\sqrt n$ at $\rho=0.3$. At $\eta=0$, calibration uses $\barchi{1}{2}$; at $\eta>0$, it uses $\chi^2_2$.  Error bars are 95\% Monte Carlo intervals.}
\label{fig:local-coverage}
\end{figure}

Figure~\ref{fig:local-coverage} reports coverage along the local sequence $d_n=\eta,\operatorname{sd}(X_2-X_1)/\sqrt n$. At $\eta=0$, calibration uses the boundary distribution $\barchi{1}{2}$, whereas for $\eta>0$ it uses $\chi^2_2$. However, for finite $\eta$, the relevant limiting distribution is $Q_\eta$, which describes the gradual transition from the boundary to the interior regime. The observed boundary-optimization rates were approximately $0.50$, $0.30$, $0.16$, $0.02$, and $0$ for $\eta=0$, $0.5$, $1$, $2$, and $4$, respectively, closely agreeing with the theoretical probabilities $\Phi(-\eta)$.

Coverage ranged from $0.926$ to $0.964$ under normal sampling, from $0.909$ to $0.957$ under Student $t_5$ sampling, and from $0.888$ to $0.949$ under shifted log-normal sampling. The greater undercoverage for heavy-tailed and skewed distributions, particularly at small sample sizes, shows that finite-sample empirical-likelihood error can dominate the conservatism predicted by the local limiting distribution. Overall, the results confirm that the transition between boundary and interior calibration is nonuniform within an $n^{-1/2}$-neighborhood of the boundary and that $Q_\eta$, rather than $\chi^2_2$, provides the appropriate local asymptotic description.
\subsection{Composite equality test and local power}
We examine the null composite equality $H_0:\mu_1=\mu_2$, under which the common mean is unspecified. By the scalar profile identity, the composite likelihood-ratio statistic reduces to the one-sided empirical-likelihood statistic for $Z_i=X_{2i}-X_{1i}$. Theorem~\ref{thm:composite}, corresponding to the bivariate ordered-mean result of \citet{el1996empirical}, gives the limiting distribution $\barchi{0}{1}=\tfrac12\chi^2_0+\tfrac12\chi^2_1$, whose 0.95 critical value is $2.706$. The simulations evaluate the finite-sample size and local power of this analytic calibration; bootstrap calibration is neither used nor required for first-order validity.

Under the composite null $H_0:\mu_1=\mu_2$, the analytic calibration $\barchi{0}{1}=\tfrac12\chi^2_0+\tfrac12\chi^2_1$, with 0.95 critical value $2.706$, produced rejection probabilities of $0.045$--$0.055$ under normal sampling, $0.051$--$0.068$ under Student $t_5$ sampling, and $0.053$--$0.070$ under shifted log-normal sampling. Thus, the calibration derived by \citet{el1996empirical} performed well under normality but was moderately liberal under heavy-tailed and skewed distributions. These results support its first-order validity while also showing that it should not be regarded as exact in small samples. In particular, profiling out the unknown common mean does not require bootstrap calibration.

\begin{figure}[!ht]
\centering
\includegraphics[width=\textwidth]{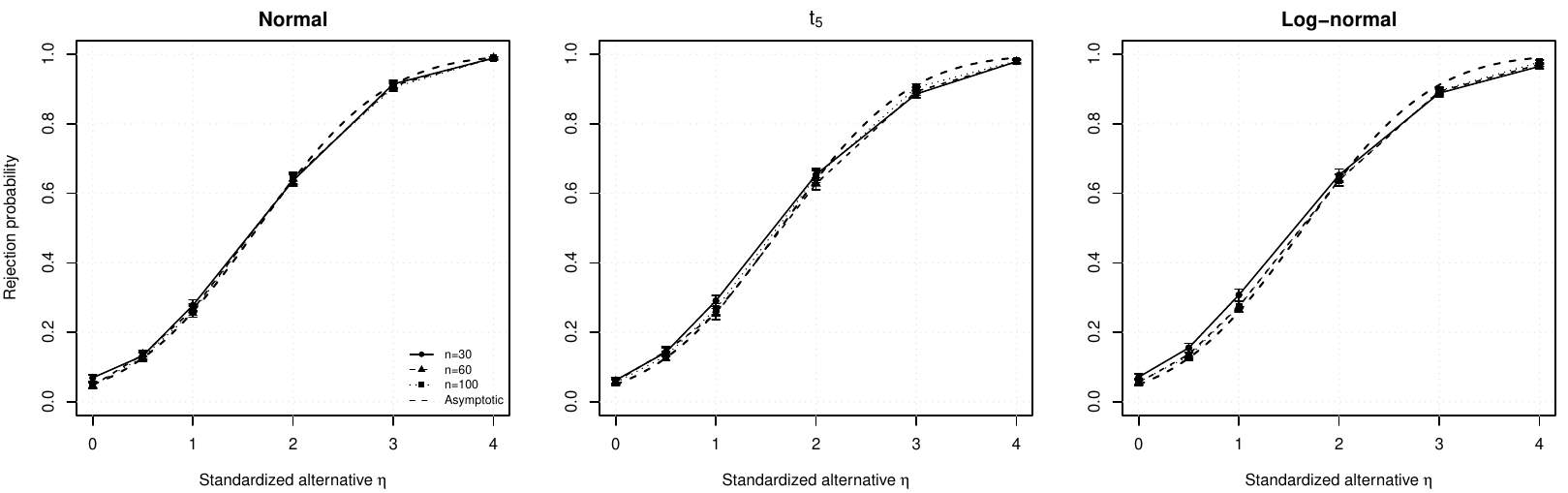}
\caption{Power of the composite equality-versus-order test under $d_n=\eta\,\operatorname{sd}(X_2-X_1)/\sqrt n$ at $\rho=0.3$. The dashed curve is the asymptotic power $1-\Phi(z_{0.95}-\eta)$.  Error bars are 95\% Monte Carlo intervals.}
\label{fig:composite-power}
\end{figure}

Figure~\ref{fig:composite-power} shows that the empirical local-power curves generally agree with the asymptotic power function. At $\eta=2$, the rejection probabilities averaged $0.649$, $0.644$, and $0.640$ across sample sizes for the normal, Student $t_5$, and shifted log-normal models, respectively, compared with the asymptotic value $0.639$. At $\eta=4$, the corresponding averages were $0.990$, $0.981$, and $0.972$, indicating that power approaches one as the standardized distance from the equality boundary increases. Overall, the composite test has the predicted local-power behavior, although convergence to the asymptotic approximation is somewhat slower under heavy-tailed and skewed sampling.

\section{Illustrative data analyses}
\label{sec:data}

We illustrate the scalar directional procedures using three paired datasets and one dataset containing precomputed change scores. For a paired observation $\boldsymbol{X}_i=(X_{1i},X_{2i})^\top$, let

$$
Z_i=X_{2i}-X_{1i},
\qquad
d=\mu_2-\mu_1=\E(Z_i).
$$

The directional hypothesis is

$$
H_0:d=0
\qquad\text{against}\qquad
H_1:d>0,
$$

or equivalently $H_0:\mu_1=\mu_2$ against $H_1:\mu_1<\mu_2$ in the paired bivariate setting. Two distinct scalar procedures are reported and should not be conflated. The confidence interval is the ordinary two-sided 95\% scalar empirical-likelihood interval intersected with the scientifically admissible parameter space,

$$
I_{0.95}^{\mathrm{tr}}
=
\left\{d\geq 0:\ell_{Z,n}(d)\leq\chi^2_{1,0.95}\right\}.
$$

It is therefore not obtained by inverting the constrained scalar likelihood ratio. The directional test uses the profiled statistic $T_n$, whose null limit is $\barchi{0}{1}=\tfrac12\chi^2_0+\tfrac12\chi^2_1$. Hence, when $\bar Z>0$, its analytic one-sided $p$-value is

$$
p_+=1-\Phi\!\left(\sqrt{T_n}\right).
$$

When $\bar Z\leq0$, the constrained statistic is $T_n=0$ and supplies no evidence in favor of $d>0$. For descriptive reporting in this case, we use the signed-root summary

$$
r_n=\operatorname{sign}(\bar Z)\sqrt{\ell_{Z,n}(0)},
\qquad
p_{\mathrm{sr}}=1-\Phi(r_n),
$$

which retains the direction of the observed difference but is not the tail $p$-value of $T_n$. For the genuinely paired datasets, we also describe the constrained bivariate empirical-likelihood estimator $\widehat{\boldsymbol{\mu}}_{\Theta}$. When the sample means satisfy the order restriction, this estimator equals the unconstrained sample mean vector; when the order is violated, it lies on the boundary $\mu_1=\mu_2$ and must be computed from the empirical-likelihood weights.

\begin{table}[htbp]
\centering
\caption{Summary of the illustrative analyses. For $\bar Z>0$, the final column reports the analytic one-sided $\barchi{0}{1}$ $p$-value $p_+$. For $\bar Z\leq0$, it reports the signed-root directional summary $p_{\mathrm{sr}}$, not the tail $p$-value of $T_n$. The Ghana subject-level quantities are naive computational summaries and are not valid inferential results because observations are clustered by health centre.}
\label{tab:real-data-results}
\small
\begin{tabular}{llrcccc}
\toprule
Dataset & Group & $n$ & $\bar Z$ & $I_{0.95}^{\mathrm{tr}}$ & $T_n$ & Directional measure\\
\midrule
Spanish scores & All participants & 20 & $1.450$ & $[0.000,\,2.756]$ & $3.782$ & $0.0259$\\
Anorexia & CBT & 29 & $3.007$ & $[0.605,\,5.913]$ & $6.154$ & $0.0066$\\
Anorexia & Control & 26 & $-0.450$ & $[0.000,\,2.698]$ & $0.000$ & $0.6145$\\
Anorexia & FT & 17 & $7.265$ & $[3.846,\,10.616]$ & $17.738$ & $1.27\times10^{-5}$\\
CalciumBP & Calcium & 10 & $5.000$ & $[0.388,\,10.194]$ & $4.663$ & $0.0154$\\
CalciumBP & Placebo & 11 & $-0.273$ & $[0.000,\,3.477]$ & $0.000$ & $0.5635$\\
Ghana SBP & Complete pairs & 641 & $18.070$ & $[16.623,\,19.521]$ & $412.066$ & $6.51\times10^{-92}$\\
\bottomrule
\end{tabular}
\end{table}

\subsection{Spanish listening scores: reproduction of El Barmi (1996)}

The first example reproduces the analysis in \citet[Example~3]{el1996empirical}, based on pre- and post-course Spanish listening scores for $n=20$ high-school teachers. We take $X_1$ and $X_2$ to be the pre- and post-course scores, respectively, so that $d$ is the mean improvement in listening score. The sample means were

$$
\bar X_{\mathrm{pre}}=27.300,
\qquad
\bar X_{\mathrm{post}}=28.750,
$$

giving $\bar Z=1.450$. Because the sample means satisfy the order restriction, the constrained estimator coincides with $(27.300,28.750)^\top$. The directional statistic was $T_n=3.782$, reproducing the value $3.78$ reported by \citet{el1996empirical} up to rounding, and the analytic one-sided $p$-value was $0.0259$.

The truncated ordinary 95\% scalar EL interval was $[0,2.756]$. Its inclusion of zero does not contradict rejection by the directional test: the interval is obtained from a two-sided 5\% ordinary EL test, whereas $T_n$ tests a one-sided alternative at level 5\%. This example therefore also illustrates why the truncated ordinary interval and the constrained directional test must be interpreted as different inferential procedures.

\subsection{Anorexia data}

The \texttt{MASS::anorexia} data contain pre- and post-study weights, measured in pounds, for 72 young female anorexia patients assigned to a control group, cognitive-behavioural treatment (CBT), or family treatment (FT) \citet[Data set~285]{hand1993handbook,venables2002modern}. We examine the mean pre--post change separately within each group, with $X_1$ denoting pre-study weight and $X_2$ denoting post-study weight. These are within-group comparisons and are not estimates of treatment contrasts between the three groups.

\paragraph{CBT group.} For the CBT group ($n=29$), the sample means were $\bar X_{\mathrm{pre}}=82.690$ and $\bar X_{\mathrm{post}}=85.697$, giving $\bar Z=3.007$. The ordering was satisfied, so the constrained and unconstrained mean estimators coincided. The truncated scalar EL interval was $[0.605,5.913]$, and the directional statistic was $T_n=6.154$, with analytic one-sided $p$-value $0.0066$. Thus, the CBT observations provide evidence of a positive mean pre--post weight change within this group.

\paragraph{Control group.} For the control group ($n=26$), the sample means were $\bar X_{\mathrm{pre}}=81.558$ and $\bar X_{\mathrm{post}}=81.108$, yielding $\bar Z=-0.450$. Because the observed means violated the assumed order, the constrained bivariate estimator lay on the equality boundary:

$$
\widehat{\boldsymbol{\mu}}_{\Theta}
=
(81.297,81.297)^\top.
$$

This common mean differs from the simple arithmetic average of the two sample means because it is determined by the empirical-likelihood weights rather than by Euclidean projection. The truncated scalar EL interval was $[0,2.698]$, and $T_n=0$. The value $0.6145$ reported in Table~\ref{tab:real-data-results} is the signed-root directional summary $p_{\mathrm{sr}}$; it is not the tail $p$-value of $T_n$. The data therefore provide no evidence that mean post-study weight exceeds mean pre-study weight in the control group.

\paragraph{FT group.} For the FT group ($n=17$), the sample means were $\bar X_{\mathrm{pre}}=83.229$ and $\bar X_{\mathrm{post}}=90.494$, giving $\bar Z=7.265$. The truncated scalar EL interval was $[3.846,10.616]$, and $T_n=17.738$, with analytic one-sided $p$-value $1.27\times10^{-5}$. These observations provide strong evidence of a positive mean pre--post weight change within the FT group.

The point estimates differ across the three groups, but these separate within-group analyses do not establish treatment-effect heterogeneity. Such a conclusion would require direct between-group contrasts, together with an appropriate multiplicity or interaction analysis. No multiplicity adjustment has been made for the three exploratory within-group tests reported here.

\subsection{Ghana blood-pressure data}

The \texttt{bp::bp\_ghana} data contain blood-pressure measurements from 757 participants in 32 community health centres participating in a pragmatic cluster-randomized trial in Ghana \citet{schwenck2022bp}. Twelve-month systolic blood pressure (SBP) was observed for 641 participants with a baseline measurement. To represent a reduction in SBP as a positive difference, we set

$$
X_1=\text{12-month SBP},
\qquad
X_2=\text{baseline SBP}.
$$

Among the 641 complete pairs, the corresponding sample means were $\bar X_1=137.824$ and $\bar X_2=155.894$ mmHg, giving an observed mean reduction of $\bar Z=18.070$ mmHg. Since the ordering was satisfied, the constrained estimator equalled $(137.824,155.894)^\top$. If the 641 pairs were treated as independent, the truncated scalar EL interval would be $[16.623,19.521]$ mmHg, with $T_n=412.066$ and analytic directional $p$-value $6.51\times10^{-92}$.

These subject-level interval and test calculations are descriptive only. Randomization occurred at the health-centre level, so participants within the same centre cannot be assumed independent, and the i.i.d. empirical-likelihood theory developed in Section~2 does not validate the reported subject-level interval or $p$-value. Moreover, the calculation combines the randomized groups and therefore does not estimate the intervention contrast. The complete-case mean reduction may also be biased if missing 12-month SBP is informative. Valid inference would require independent centre-level summaries or a cluster-robust empirical-likelihood extension, together with an appropriate treatment of missing follow-up measurements.

\subsection{Calcium and blood-pressure data}

The \texttt{Stat2Data::CalciumBP} data contain one blood-pressure decrease score for each of 21 men randomized to calcium supplementation or placebo \citet{lyle1987blood}. The recorded response is beginning-minus-ending blood pressure over the 12-week study period, so a positive value represents a reduction in blood pressure. Because only the derived decrease score is available, this example illustrates the scalar $Z$-based procedure rather than inference for a bivariate mean vector.

For the calcium group ($n=10$), the mean decrease was $5.000$. The truncated scalar EL interval was $[0.388,10.194]$, and the directional statistic was $T_n=4.663$, with analytic one-sided $p$-value $0.0154$. Thus, the calcium group shows evidence of a positive mean decrease relative to zero. For the placebo group ($n=11$), the mean decrease was $-0.273$, the truncated interval was $[0,3.477]$, and $T_n=0$. The reported value $0.5635$ is the signed-root directional summary and indicates no positive directional evidence within the placebo group.




\subsection{Summary of the data illustrations}

The examples illustrate how the procedures behave when the empirical ordering is either satisfied or violated. When $\bar Z>0$, as in the Spanish, CBT, FT, and complete-case Ghana calculations, the bivariate constrained estimator coincides with the unconstrained sample mean vector. When $\bar Z<0$, as in the anorexia control group, the constrained bivariate estimator lies on the equality boundary; importantly, its common coordinate is determined by empirical-likelihood weighting and is not generally the Euclidean average of the sample means. The scalar confidence interval remains within $[0,\infty)$ because the ordinary two-sided EL interval is intersected with the admissible parameter space after inversion.

These analyses are methodological illustrations rather than evidence that the order restriction itself generates additional information. Their interpretation must respect the underlying sampling designs: the anorexia calculations are exploratory within-group comparisons, the Ghana subject-level calculations do not account for centre-level clustering or potentially informative missingness, and the CalciumBP within-arm analyses do not estimate the randomized treatment contrast. Subject to these qualifications, the examples demonstrate the computational and interpretive distinctions among constrained bivariate estimation, truncated ordinary scalar EL intervals, and analytic one-sided equality-versus-order testing.

\section{Discussion}

The central conceptual distinction is between fixed-vector and profiled boundary hypotheses. General chi-bar-square theory for inequality-constrained empirical likelihood already exists in El Barmi~\cite{el1996empirical}. The composite equality-versus-order test considered in this paper is precisely its bivariate ordered-mean specialization and should therefore use the analytic $\bar{\chi}^{2}_{0,1}$ limiting distribution. The fixed-vector boundary statistic is different because it fixes both coordinates, leaving one tangential degree of freedom when the order constraint is active; its limiting distribution is $\bar{\chi}^{2}_{1,2}$.

The local result clarifies why a binary interior--boundary description is not uniform. For distances of order $n^{-1/2}$ from the boundary, the limiting distribution is the interpolating law $Q_{\eta}$, rather than either endpoint distribution, except when $\eta=0$ or as $\eta\to\infty$. Candidate-wise confidence regions are pointwise valid, but their finite-sample coverage may be poor under heavy-tailed or skewed distributions. The simulation results therefore do not support the claim that first-order chi-bar-square calibration alone resolves small-sample approximation error. Adjusted or bootstrap calibration may be useful as a finite-sample extension, but any such procedure requires its own theoretical or empirical justification and is not necessary merely because the common boundary mean is unknown.

\section*{Declaration of competing interest}
The author declares no known competing financial interests or personal relationships that could have appeared to influence the work reported in this paper.

\section*{Data and code availability}

The public data sources, scripts, complete design grid, and files needed to reproduce all tables and figures are available in the project repository.\footnote{\url{https://github.com/nareshng/Empirical-likelihood-confidence-regions-for-ordered-means}}

\section*{Acknowledgements}

The author thanks the Editor and anonymous reviewers for their constructive comments and suggestions, which helped improve the presentation and clarity of the manuscript.

\bibliographystyle{elsarticle-harv}
\bibliography{sample}

\section{\textbf{Appendix A}} \label{app:proofs}

\subsection{Proof of Lemma~\ref{lem:quadratic}}

Write $\boldsymbol{Y}_i=\boldsymbol{X}_i-\boldsymbol{\mu}_0$ and $\boldsymbol{W}_i(\boldsymbol{h})=\boldsymbol{Y}_i-\boldsymbol{h}/\sqrt{n}$. Define

$$ \overline{\boldsymbol{W}}_n(\boldsymbol{h})=\frac{1}{n}\sum_{i=1}^n\boldsymbol{W}_i(\boldsymbol{h})=\frac{1}{\sqrt{n}}(\boldsymbol{\Delta}_n-\boldsymbol{h}) $$

and

$$ \boldsymbol{S}_n(\boldsymbol{h})=\frac{1}{n}\sum_{i=1}^n\boldsymbol{W}_i(\boldsymbol{h})\boldsymbol{W}_i(\boldsymbol{h})^\top. $$

Uniformly over $\|\boldsymbol{h}\|\leq M$, the law of large numbers and (B3) give

$$ \sup_{\|\boldsymbol{h}\|\leq M}\|\boldsymbol{S}_n(\boldsymbol{h})-\boldsymbol{\Sigma}\|=o_p(1), $$

$$ \sup_{\|\boldsymbol{h}\|\leq M}\frac{1}{n}\sum_{i=1}^n\|\boldsymbol{W}_i(\boldsymbol{h})\|^3=O_p(1), $$

and

$$ \sup_{\|\boldsymbol{h}\|\leq M}\max_{1\leq i\leq n}\|\boldsymbol{W}_i(\boldsymbol{h})\|=o_p(n^{1/3}). $$

Moreover,

$$ \sup_{\|\boldsymbol{h}\|\leq M}\|\overline{\boldsymbol{W}}_n(\boldsymbol{h})\|=O_p(n^{-1/2}). $$

On the event in (B4), let $\boldsymbol{\lambda}_n(\boldsymbol{h})$ denote the solution of the empirical-likelihood multiplier equation

$$ \sum_{i=1}^n\frac{\boldsymbol{W}_i(\boldsymbol{h})}{1+\boldsymbol{\lambda}_n(\boldsymbol{h})^\top\boldsymbol{W}_i(\boldsymbol{h})}=\boldsymbol{0}. $$

We first establish the usual multiplier-localization bound uniformly over $\|\boldsymbol{h}\|\leq M$; see \citet{owen1990}, \citet{qinlawless1994}, and \citet{owen2001}. Put

$$ b_n(\boldsymbol{h})=\|\overline{\boldsymbol{W}}_n(\boldsymbol{h})\|,\qquad m_n(\boldsymbol{h})=\max_{1\leq i\leq n}\|\boldsymbol{W}_i(\boldsymbol{h})\|,\qquad r_n(\boldsymbol{h})=\|\boldsymbol{\lambda}_n(\boldsymbol{h})\|. $$

If $r_n(\boldsymbol{h})>0$, take the inner product of the multiplier equation with $\boldsymbol{\lambda}_n(\boldsymbol{h})/r_n(\boldsymbol{h})$. Using $x/(1+rx)=x-rx^2/(1+rx)$ gives

$$ b_n(\boldsymbol{h})\geq\frac{r_n(\boldsymbol{h})\lambda_{\min}\{\boldsymbol{S}_n(\boldsymbol{h})\}}{1+r_n(\boldsymbol{h})m_n(\boldsymbol{h})}. $$

The preceding bounds imply

$$ \sup_{\|\boldsymbol{h}\|\leq M}b_n(\boldsymbol{h})m_n(\boldsymbol{h})=o_p(1). $$

In addition, the infimum over $\|\boldsymbol{h}\|\leq M$ of the smallest eigenvalue of $\boldsymbol{S}_n(\boldsymbol{h})$ is bounded away from zero with probability tending to one. Rearranging the preceding inequality therefore gives

$$ \sup_{\|\boldsymbol{h}\|\leq M}\|\boldsymbol{\lambda}_n(\boldsymbol{h})\|=O_p(n^{-1/2}), $$

and hence

$$ \sup_{\|\boldsymbol{h}\|\leq M}\max_{1\leq i\leq n}\left|\boldsymbol{\lambda}_n(\boldsymbol{h})^\top\boldsymbol{W}_i(\boldsymbol{h})\right|=o_p(1). $$

Using $(1+t)^{-1}=1-t+t^2/(1+t)$ in the multiplier equation gives

$$ \boldsymbol{0}=\overline{\boldsymbol{W}}_n(\boldsymbol{h})-\boldsymbol{S}_n(\boldsymbol{h})\boldsymbol{\lambda}_n(\boldsymbol{h})+\boldsymbol{R}_n(\boldsymbol{h}), $$

where

$$ \boldsymbol{R}_n(\boldsymbol{h})=\frac{1}{n}\sum_{i=1}^n\frac{\boldsymbol{W}_i(\boldsymbol{h})\{\boldsymbol{\lambda}_n(\boldsymbol{h})^\top\boldsymbol{W}_i(\boldsymbol{h})\}^2}{1+\boldsymbol{\lambda}_n(\boldsymbol{h})^\top\boldsymbol{W}_i(\boldsymbol{h})}. $$

The preceding bounds yield

$$ \sup_{\|\boldsymbol{h}\|\leq M}\|\boldsymbol{R}_n(\boldsymbol{h})\|=O_p(n^{-1}). $$

Since the smallest eigenvalue of $\boldsymbol{S}_n(\boldsymbol{h})$ is bounded away from zero with probability tending to one, it follows that

$$ \sup_{\|\boldsymbol{h}\|\leq M}\left\|\boldsymbol{\lambda}_n(\boldsymbol{h})-\frac{1}{\sqrt{n}}\boldsymbol{\Sigma}^{-1}(\boldsymbol{\Delta}_n-\boldsymbol{h})\right\|=o_p(n^{-1/2}). $$

Put

$$ t_{ni}(\boldsymbol{h})=\boldsymbol{\lambda}_n(\boldsymbol{h})^\top\boldsymbol{W}_i(\boldsymbol{h}). $$

Since

$$ \sum_{i=1}^n t_{ni}^2(\boldsymbol{h})=n\boldsymbol{\lambda}_n(\boldsymbol{h})^\top\boldsymbol{S}_n(\boldsymbol{h})\boldsymbol{\lambda}_n(\boldsymbol{h}), $$

we have

$$ \sup_{\|\boldsymbol{h}\|\leq M}\sum_{i=1}^nt_{ni}^2(\boldsymbol{h})=O_p(1). $$

Consequently,

$$ \sup_{\|\boldsymbol{h}\|\leq M}\sum_{i=1}^n|t_{ni}(\boldsymbol{h})|^3\leq\left\{\sup_{\|\boldsymbol{h}\|\leq M}\max_{1\leq i\leq n}|t_{ni}(\boldsymbol{h})|\right\}\left\{\sup_{\|\boldsymbol{h}\|\leq M}\sum_{i=1}^nt_{ni}^2(\boldsymbol{h})\right\}=o_p(1). $$

The multiplier equation also gives

$$ \sum_{i=1}^nt_{ni}(\boldsymbol{h})=\sum_{i=1}^nt_{ni}^2(\boldsymbol{h})+o_p(1) $$

uniformly over $\|\boldsymbol{h}\|\leq M$. Using $\log(1+t)=t-\tfrac12t^2+O(|t|^3)$, we obtain

$$ \ell_n\left(\boldsymbol{\mu}_0+\frac{\boldsymbol{h}}{\sqrt{n}}\right)=2\sum_{i=1}^n\log\{1+t_{ni}(\boldsymbol{h})\}=n\boldsymbol{\lambda}_n(\boldsymbol{h})^\top\boldsymbol{S}_n(\boldsymbol{h})\boldsymbol{\lambda}_n(\boldsymbol{h})+o_p(1)=(\boldsymbol{\Delta}_n-\boldsymbol{h})^\top\boldsymbol{\Sigma}^{-1}(\boldsymbol{\Delta}_n-\boldsymbol{h})+o_p(1), $$

uniformly over $\|\boldsymbol{h}\|\leq M$. Therefore,

$$ \sup_{\|\boldsymbol{h}\|\leq M}\left|\ell_n\left(\boldsymbol{\mu}_0+\frac{\boldsymbol{h}}{\sqrt{n}}\right)-(\boldsymbol{\Delta}_n-\boldsymbol{h})^\top\boldsymbol{\Sigma}^{-1}(\boldsymbol{\Delta}_n-\boldsymbol{h})\right|\xrightarrow{\Pp}0. $$

\subsection{Proof of Lemma~\ref{lem:localization}}

Recall that

$$ Z_i=\boldsymbol{a}^\top\boldsymbol{X}_i=X_{2i}-X_{1i},\qquad \overline{Z}=\frac{1}{n}\sum_{i=1}^nZ_i. $$

Because $\boldsymbol{\mu}_0\in H$,

$$ \E(Z_i)=0,\qquad \operatorname{Var}(Z_i)=\boldsymbol{a}^\top\boldsymbol{\Sigma}\boldsymbol{a}>0. $$

Assumption (B4), applied with $M=0$, implies that $\boldsymbol{\mu}_0$ belongs to the interior of the sample convex hull with probability tending to one. Hence both $\mathcal{D}_{\Theta,n}$ and $\mathcal{D}_{H,n}$ occur with probability tending to one. The arbitrary definitions of the two maximizers on the complementary events therefore do not affect their stochastic orders.

Suppose first that $\overline{Z}\geq0$. The uniform weights are then feasible for $\Theta$, and the unrestricted empirical-likelihood maximum is attained at $\overline{\boldsymbol{X}}$. Thus,

$$ \widehat{\boldsymbol{\mu}}_\Theta=\overline{\boldsymbol{X}}, $$

and the central limit theorem gives

$$ \widehat{\boldsymbol{\mu}}_\Theta-\boldsymbol{\mu}_0=O_p(n^{-1/2}). $$

Now suppose that $\overline{Z}<0$. Since $\E(Z_i)=0$ and $\operatorname{Var}(Z_i)>0$, the sample contains both positive and negative values of $Z_i$ with probability tending to one. By Lemma~\ref{lem:profile} and the concavity of $d\mapsto\log L_{Z,n}(d)$, the order restriction binds, and the maximizing weights are

$$ p_i^0=\frac{1}{n(1+\gamma_nZ_i)},\qquad i=1,\ldots,n, $$

where $\gamma_n$ satisfies

$$ \sum_{i=1}^n\frac{Z_i}{1+\gamma_nZ_i}=0. $$

The standard scalar empirical-likelihood multiplier expansion \citet{owen2001} gives

$$ \gamma_n=\frac{\overline{Z}}{n^{-1}\sum_{i=1}^nZ_i^2}+o_p(n^{-1/2})=O_p(n^{-1/2}),\qquad \max_{1\leq i\leq n}|\gamma_nZ_i|=o_p(1). $$

The constrained maximizer is the weighted mean corresponding to these weights. Therefore,

$$ \widehat{\boldsymbol{\mu}}_\Theta-\boldsymbol{\mu}_0=\frac{1}{n}\sum_{i=1}^n\frac{\boldsymbol{X}_i-\boldsymbol{\mu}_0}{1+\gamma_nZ_i}=\overline{\boldsymbol{X}}-\boldsymbol{\mu}_0-\gamma_n\frac{1}{n}\sum_{i=1}^n\frac{(\boldsymbol{X}_i-\boldsymbol{\mu}_0)Z_i}{1+\gamma_nZ_i}. $$

Because $\max_i|\gamma_nZ_i|=o_p(1)$ and

$$ \frac{1}{n}\sum_{i=1}^n\|\boldsymbol{X}_i-\boldsymbol{\mu}_0\|\,|Z_i|=O_p(1), $$

we have

$$ \frac{1}{n}\sum_{i=1}^n\frac{(\boldsymbol{X}_i-\boldsymbol{\mu}_0)Z_i}{1+\gamma_nZ_i}=O_p(1). $$

It follows that

$$ \widehat{\boldsymbol{\mu}}_\Theta-\boldsymbol{\mu}_0=O_p(n^{-1/2}). $$

For the maximizer over $H$, Lemma~\ref{lem:profile} reduces the problem to the scalar constraint $\E(Z_i)=0$, regardless of the sign of $\overline{Z}$. On the event that the sample contains both positive and negative values of $Z_i$, the maximizing weights have the same form

$$ p_i^0=\frac{1}{n(1+\gamma_nZ_i)}, $$

where

$$ \gamma_n=O_p(n^{-1/2}),\qquad \max_{1\leq i\leq n}|\gamma_nZ_i|=o_p(1). $$

Repeating the preceding calculation gives

$$ \widehat{\boldsymbol{\mu}}_H-\boldsymbol{\mu}_0=O_p(n^{-1/2}). $$

This proves both localization statements.

\subsection{Proof of Lemma~\ref{lem:reduction}}

Define

$$ \widehat{\boldsymbol{h}}_{\Theta,n}=\sqrt{n}\left(\widehat{\boldsymbol{\mu}}_\Theta-\boldsymbol{\mu}_0\right). $$

By Lemma~\ref{lem:localization},

$$ \widehat{\boldsymbol{h}}_{\Theta,n}=O_p(1). $$

Let

$$ q_n(\boldsymbol{h})=(\boldsymbol{\Delta}_n-\boldsymbol{h})^\top\boldsymbol{\Sigma}^{-1}(\boldsymbol{\Delta}_n-\boldsymbol{h}), $$

and let

$$ \widetilde{\boldsymbol{h}}_n\in\argmin_{\boldsymbol{h}\in K}q_n(\boldsymbol{h}). $$

Thus, $\widetilde{\boldsymbol{h}}_n$ is the $\boldsymbol{\Sigma}^{-1}$-metric projection of $\boldsymbol{\Delta}_n$ onto $K$. Since $\boldsymbol{0}\in K$,

$$ q_n(\widetilde{\boldsymbol{h}}_n)\leq q_n(\boldsymbol{0})=\boldsymbol{\Delta}_n^\top\boldsymbol{\Sigma}^{-1}\boldsymbol{\Delta}_n. $$

Because $\boldsymbol{\Delta}_n=O_p(1)$ and $\boldsymbol{\Sigma}$ is positive definite,

$$ \widetilde{\boldsymbol{h}}_n=O_p(1). $$

For $\boldsymbol{h}\in K$, define

$$ Q_n(\boldsymbol{h})=\ell_n\left(\boldsymbol{\mu}_0+\frac{\boldsymbol{h}}{\sqrt{n}}\right), $$

with $Q_n(\boldsymbol{h})=+\infty$ whenever the candidate mean lies outside the empirical-likelihood domain. Since $\boldsymbol{\mu}_0\in H$,

$$ \boldsymbol{\mu}_0+\frac{\boldsymbol{h}}{\sqrt{n}}\in\Theta\quad\Longleftrightarrow\quad\boldsymbol{h}\in K. $$

Consequently,

$$ \inf_{\boldsymbol{\mu}\in\Theta}\ell_n(\boldsymbol{\mu})=\inf_{\boldsymbol{h}\in K}Q_n(\boldsymbol{h}). $$

For every $\varepsilon>0$, choose a deterministic $M<\infty$ such that

$$ \Pp\left(\|\widehat{\boldsymbol{h}}_{\Theta,n}\|\leq M,\ \|\widetilde{\boldsymbol{h}}_n\|\leq M\right)\geq1-\varepsilon $$

for all sufficiently large $n$. Put

$$ r_{n,M}=\sup_{\substack{\boldsymbol{h}\in K\\\|\boldsymbol{h}\|\leq M}}|Q_n(\boldsymbol{h})-q_n(\boldsymbol{h})|. $$

By Lemma~\ref{lem:quadratic},

$$ r_{n,M}=o_p(1). $$

On the event that both minimizers lie in the ball of radius $M$, their respective optimality properties give

$$ q_n(\widetilde{\boldsymbol{h}}_n)-r_{n,M}\leq Q_n(\widehat{\boldsymbol{h}}_{\Theta,n})\leq q_n(\widetilde{\boldsymbol{h}}_n)+r_{n,M}. $$

It follows that

$$ \inf_{\boldsymbol{\mu}\in\Theta}\ell_n(\boldsymbol{\mu})=\inf_{\boldsymbol{h}\in K}(\boldsymbol{\Delta}_n-\boldsymbol{h})^\top\boldsymbol{\Sigma}^{-1}(\boldsymbol{\Delta}_n-\boldsymbol{h})+o_p(1). $$

Applying Lemma~\ref{lem:quadratic} at $\boldsymbol{h}=\boldsymbol{0}$ also gives

$$ \ell_n(\boldsymbol{\mu}_0)=\boldsymbol{\Delta}_n^\top\boldsymbol{\Sigma}^{-1}\boldsymbol{\Delta}_n+o_p(1). $$

On $\mathcal{D}_{\Theta,n}$, the definition of the constrained empirical log-likelihood ratio gives

$$ \ell_{\Theta,n}(\boldsymbol{\mu}_0)=\ell_n(\boldsymbol{\mu}_0)-\inf_{\boldsymbol{\mu}\in\Theta}\ell_n(\boldsymbol{\mu}). $$

Since $\Pp(\mathcal{D}_{\Theta,n})\to1$, substitution yields

$$ \ell_{\Theta,n}(\boldsymbol{\mu}_0)=\boldsymbol{\Delta}_n^\top\boldsymbol{\Sigma}^{-1}\boldsymbol{\Delta}_n-\inf_{\boldsymbol{h}\in K}(\boldsymbol{\Delta}_n-\boldsymbol{h})^\top\boldsymbol{\Sigma}^{-1}(\boldsymbol{\Delta}_n-\boldsymbol{h})+o_p(1), $$

which proves \eqref{eq:projection-reduction}.

\subsection{Proof of Proposition~\ref{prop:local}}
\begin{proof}
Put $\overline{\boldsymbol{X}}_n=n^{-1}\sum_{i=1}^n\boldsymbol{X}_{ni}$ and
$$\boldsymbol{\Delta}_n=\sqrt{n}\bigl(\overline{\boldsymbol{X}}_n-\boldsymbol{\mu}_n\bigr)=n^{-1/2}\sum_{i=1}^n\boldsymbol{\varepsilon}_{ni}.$$
Then $\boldsymbol{\Delta}_n\Rightarrow\boldsymbol{G}\sim N_2(\boldsymbol{0},\boldsymbol{\Sigma})$, and hence $\boldsymbol{\Delta}_n=O_p(1)$. Define
$$b_n=\sqrt{n}\,\boldsymbol{a}^{\top}\boldsymbol{\mu}_n.$$
For a local parameter $\boldsymbol{\nu}=\boldsymbol{\mu}_n+\boldsymbol{h}/\sqrt{n}$, the order constraint is equivalent to $\boldsymbol{h}\in K_n$, where
$$K_n=\left\{\boldsymbol{h}\in\mathbb{R}^2:\boldsymbol{a}^{\top}\boldsymbol{h}\geq-b_n\right\}.$$
Because $\boldsymbol{X}_{ni}-\boldsymbol{\mu}_n=\boldsymbol{\varepsilon}_{ni}$ has a distribution that does not depend on $n$, the proof of Lemma~\ref{lem:quadratic} applies row by row. Therefore, for every fixed $M<\infty$,
$$\sup_{\|\boldsymbol{h}\|\leq M}\left|\ell_n\left(\boldsymbol{\mu}_n+\frac{\boldsymbol{h}}{\sqrt{n}}\right)-(\boldsymbol{\Delta}_n-\boldsymbol{h})^{\top}\boldsymbol{\Sigma}^{-1}(\boldsymbol{\Delta}_n-\boldsymbol{h})\right|\xrightarrow{\mathbb{P}}0.$$
The assumed local convex-hull condition also gives $\mathbb{P}(\mathcal{D}_{\Theta,n})\to1$. Thus, the arbitrary definition of the constrained maximizer on $\mathcal{D}_{\Theta,n}^{c}$ does not affect the following stochastic-order arguments.

It remains to establish localization under the moving constraint. Let
$$Z_{ni}=\boldsymbol{a}^{\top}\boldsymbol{X}_{ni}=d_n+\xi_{ni},\qquad \overline{Z}_n=\frac{1}{n}\sum_{i=1}^nZ_{ni},$$
where $d_n=\boldsymbol{a}^{\top}\boldsymbol{\mu}_n=O(n^{-1/2})$ and $\xi_{ni}=\boldsymbol{a}^{\top}\boldsymbol{\varepsilon}_{ni}$ has mean zero and variance
$$\sigma_Z^2=\boldsymbol{a}^{\top}\boldsymbol{\Sigma}\boldsymbol{a}>0.$$
Notice that
$$\sqrt{n}\,\overline{Z}_n=b_n+\boldsymbol{a}^{\top}\boldsymbol{\Delta}_n.$$
Let $\widehat{\boldsymbol{\mu}}_{\Theta,n}$ be a constrained maximizer on $\mathcal{D}_{\Theta,n}$ and set it equal to $\boldsymbol{\mu}_n$ on the complement. If $\overline{Z}_n\geq0$, the uniform weights are feasible, and the constrained maximizer is $\overline{\boldsymbol{X}}_n$. Consequently,
$$\sqrt{n}\bigl(\widehat{\boldsymbol{\mu}}_{\Theta,n}-\boldsymbol{\mu}_n\bigr)=\boldsymbol{\Delta}_n=O_p(1).$$
Suppose now that $\overline{Z}_n<0$. The profile identity and strict concavity of the scalar empirical log-likelihood imply that the order constraint binds at zero. Moreover, with probability tending to one, the sample contains values of $Z_{ni}$ on both sides of zero. On this event, the maximizing weights are
$$p_{ni}^{0}=\frac{1}{n(1+\gamma_nZ_{ni})},\qquad \sum_{i=1}^n\frac{Z_{ni}}{1+\gamma_nZ_{ni}}=0.$$
Repeating the scalar multiplier argument in the proof of Lemma~\ref{lem:localization}, using $\overline{Z}_n=O_p(n^{-1/2})$, $n^{-1}\sum_{i=1}^nZ_{ni}^2\xrightarrow{\mathbb{P}}\sigma_Z^2$, and $\max_{1\leq i\leq n}|Z_{ni}|/\sqrt{n}=o_p(1)$, gives
$$\gamma_n=\frac{\overline{Z}_n}{n^{-1}\sum_{i=1}^nZ_{ni}^2}+o_p(n^{-1/2})=O_p(n^{-1/2}),\qquad \max_{1\leq i\leq n}|\gamma_nZ_{ni}|=o_p(1).$$
Equivalently, on the binding event,
$$\sqrt{n}\,\gamma_n=\frac{\boldsymbol{a}^{\top}\boldsymbol{\Delta}_n+b_n}{\sigma_Z^2}+o_p(1).$$
The multiplier equation implies $n^{-1}\sum_{i=1}^n(1+\gamma_nZ_{ni})^{-1}=1$. Put $\overline{\boldsymbol{\varepsilon}}_n=n^{-1}\sum_{i=1}^n\boldsymbol{\varepsilon}_{ni}$. Since the weights sum to one,
$$\widehat{\boldsymbol{\mu}}_{\Theta,n}-\boldsymbol{\mu}_n=\frac{1}{n}\sum_{i=1}^n\frac{\boldsymbol{\varepsilon}_{ni}}{1+\gamma_nZ_{ni}}=\overline{\boldsymbol{\varepsilon}}_n-\gamma_n\frac{1}{n}\sum_{i=1}^n\frac{\boldsymbol{\varepsilon}_{ni}Z_{ni}}{1+\gamma_nZ_{ni}}.$$
The third-moment condition and $\max_i|\gamma_nZ_{ni}|=o_p(1)$ yield
$$\frac{1}{n}\sum_{i=1}^n\frac{\boldsymbol{\varepsilon}_{ni}Z_{ni}}{1+\gamma_nZ_{ni}}=\boldsymbol{\Sigma}\boldsymbol{a}+o_p(1)=O_p(1).$$
Therefore,
$$\widehat{\boldsymbol{\mu}}_{\Theta,n}-\boldsymbol{\mu}_n=O_p(n^{-1/2}),\qquad \widehat{\boldsymbol{h}}_{\Theta,n}:=\sqrt{n}\bigl(\widehat{\boldsymbol{\mu}}_{\Theta,n}-\boldsymbol{\mu}_n\bigr)=O_p(1).$$

Define
$$q_n(\boldsymbol{h})=(\boldsymbol{\Delta}_n-\boldsymbol{h})^{\top}\boldsymbol{\Sigma}^{-1}(\boldsymbol{\Delta}_n-\boldsymbol{h}),$$
and let
$$\widetilde{\boldsymbol{h}}_n=\operatorname*{arg\,min}_{\boldsymbol{h}\in K_n}q_n(\boldsymbol{h}).$$
The minimizer is the $\boldsymbol{\Sigma}^{-1}$-metric projection of $\boldsymbol{\Delta}_n$ onto $K_n$ and has the explicit representation
$$\widetilde{\boldsymbol{h}}_n=\boldsymbol{\Delta}_n-\frac{\boldsymbol{a}^{\top}\boldsymbol{\Delta}_n+b_n}{\sigma_Z^2}\boldsymbol{\Sigma}\boldsymbol{a}\,\mathbf{1}\left\{\boldsymbol{a}^{\top}\boldsymbol{\Delta}_n+b_n<0\right\}.$$
Because $\boldsymbol{\Delta}_n=O_p(1)$ and $b_n=O(1)$, it follows that $\widetilde{\boldsymbol{h}}_n=O_p(1)$. The preceding multiplier expansion also shows that
$$\widehat{\boldsymbol{h}}_{\Theta,n}-\widetilde{\boldsymbol{h}}_n=o_p(1).$$

For $\boldsymbol{h}\in K_n$, define
$$Q_n(\boldsymbol{h})=\ell_n\left(\boldsymbol{\mu}_n+\frac{\boldsymbol{h}}{\sqrt{n}}\right),$$
with $Q_n(\boldsymbol{h})=+\infty$ whenever the candidate mean lies outside the empirical-likelihood domain. For every $\epsilon>0$, the localization results allow us to choose a fixed $M<\infty$ such that
$$\limsup_{n\to\infty}\mathbb{P}\left(\|\widehat{\boldsymbol{h}}_{\Theta,n}\|>M\ \text{or}\ \|\widetilde{\boldsymbol{h}}_n\|>M\right)<\epsilon.$$
On the event that both minimizers belong to this ball, put
$$r_{n,M}=\sup_{\substack{\boldsymbol{h}\in K_n\\ \|\boldsymbol{h}\|\leq M}}|Q_n(\boldsymbol{h})-q_n(\boldsymbol{h})|.$$
The uniform quadratic expansion gives $r_{n,M}=o_p(1)$. Evaluating each criterion at the minimizer of the other criterion then gives
$$\left|\inf_{\boldsymbol{h}\in K_n}Q_n(\boldsymbol{h})-\inf_{\boldsymbol{h}\in K_n}q_n(\boldsymbol{h})\right|\leq r_{n,M}=o_p(1).$$
Because $\epsilon>0$ is arbitrary,
$$\inf_{\boldsymbol{\mu}\in\Theta}\ell_n(\boldsymbol{\mu})=\inf_{\boldsymbol{h}\in K_n}(\boldsymbol{\Delta}_n-\boldsymbol{h})^{\top}\boldsymbol{\Sigma}^{-1}(\boldsymbol{\Delta}_n-\boldsymbol{h})+o_p(1).$$
The local quadratic expansion at $\boldsymbol{h}=\boldsymbol{0}$ gives
$$\ell_n(\boldsymbol{\mu}_n)=\boldsymbol{\Delta}_n^{\top}\boldsymbol{\Sigma}^{-1}\boldsymbol{\Delta}_n+o_p(1).$$
On $\mathcal{D}_{\Theta,n}$,
$$\ell_{\Theta,n}(\boldsymbol{\mu}_n)=\ell_n(\boldsymbol{\mu}_n)-\inf_{\boldsymbol{\mu}\in\Theta}\ell_n(\boldsymbol{\mu}).$$
Since $\mathbb{P}(\mathcal{D}_{\Theta,n})\to1$, substitution yields
$$\ell_{\Theta,n}(\boldsymbol{\mu}_n)=\boldsymbol{\Delta}_n^{\top}\boldsymbol{\Sigma}^{-1}\boldsymbol{\Delta}_n-\inf_{\boldsymbol{h}\in K_n}(\boldsymbol{\Delta}_n-\boldsymbol{h})^{\top}\boldsymbol{\Sigma}^{-1}(\boldsymbol{\Delta}_n-\boldsymbol{h})+o_p(1).$$

Let $\boldsymbol{\Sigma}^{1/2}$ be the symmetric positive-definite square root of $\boldsymbol{\Sigma}$, and set
$$\boldsymbol{Y}_n=\boldsymbol{\Sigma}^{-1/2}\boldsymbol{\Delta}_n,\qquad \boldsymbol{c}=\boldsymbol{\Sigma}^{1/2}\boldsymbol{a},\qquad \sigma_Z=\|\boldsymbol{c}\|,\qquad \boldsymbol{e}_1=\frac{\boldsymbol{c}}{\sigma_Z}.$$
Choose a deterministic unit vector $\boldsymbol{e}_2$ orthogonal to $\boldsymbol{e}_1$ and define
$$U_n=\boldsymbol{e}_1^{\top}\boldsymbol{Y}_n=\frac{\boldsymbol{a}^{\top}\boldsymbol{\Delta}_n}{\sigma_Z},\qquad V_n=\boldsymbol{e}_2^{\top}\boldsymbol{Y}_n.$$
Then $(U_n,V_n)\Rightarrow(U,V)$, where $U$ and $V$ are independent standard normal variables. Under the change of variable $\boldsymbol{x}=\boldsymbol{\Sigma}^{-1/2}\boldsymbol{h}$, the set $K_n$ becomes
$$\left\{\boldsymbol{x}\in\mathbb{R}^2:\boldsymbol{e}_1^{\top}\boldsymbol{x}\geq-\eta_n\right\},\qquad \eta_n=\frac{b_n}{\sigma_Z}=\frac{\sqrt{n}\,\boldsymbol{a}^{\top}\boldsymbol{\mu}_n}{\sigma_Z}.$$
Consequently, the squared projection loss is
$$\inf_{\boldsymbol{h}\in K_n}q_n(\boldsymbol{h})=(U_n+\eta_n)^2\mathbf{1}\{U_n<-\eta_n\},$$
and hence
$$\ell_{\Theta,n}(\boldsymbol{\mu}_n)=U_n^2+V_n^2-(U_n+\eta_n)^2\mathbf{1}\{U_n<-\eta_n\}+o_p(1).$$
By assumption, $\eta_n\to\eta$, so
$$(U_n,V_n,\eta_n)\Rightarrow(U,V,\eta).$$
Writing $x_-=\max\{-x,0\}$, the function
$$f(u,v,t)=u^2+v^2-\{(u+t)_-\}^2$$
is continuous everywhere, including at $u=-t$. The continuous mapping theorem therefore gives
$$\ell_{\Theta,n}(\boldsymbol{\mu}_n)\Rightarrow Q_\eta=U^2+V^2-(U+\eta)^2\mathbf{1}\{U<-\eta\}.$$
At $\eta=0$,
$$Q_0=\begin{cases}U^2+V^2,&U\geq0,\\V^2,&U<0.\end{cases}$$
Because $\operatorname{sign}(U)$ is independent of $U^2$, and $V$ is independent of $U$, it follows that
$$Q_0\sim\frac{1}{2}\chi_1^2+\frac{1}{2}\chi_2^2.$$
Finally, under a common coupling of $U$ and $V$, $\mathbf{1}\{U<-\eta\}\to0$ almost surely as $\eta\to\infty$. Therefore,
$$Q_\eta\longrightarrow U^2+V^2\quad\text{almost surely},$$
and hence $Q_\eta\Rightarrow\chi_2^2$ as $\eta\to\infty$.
\end{proof}

\section{\textbf{Appendix B}} \label{app:simulation}

\begin{longtable}{llrrrrr}
\caption{Interior results for $\mu=(0,1)^\top$. Coverage uses the $\chi^2_2$ critical value 5.991. Each entry uses 5,000 replications.}\label{tab:interior-full}\\
\toprule
Model & $\rho$ & $n$ & Coverage & MCSE & Empirical $q_{.95}$ & Boundary rate \\
\midrule
\endfirsthead
\toprule
Model & $\rho$ & $n$ & Coverage & MCSE & Empirical $q_{.95}$ & Boundary rate \\
\midrule
\endhead
\midrule
\multicolumn{7}{r}{Continued on next page}\\
\endfoot
\bottomrule
\endlastfoot
Log-normal & -0.5 & 30 & 0.898 & 0.004 & 8.530 & 0.001 \\
Log-normal & -0.5 & 60 & 0.934 & 0.004 & 6.693 & 0.000 \\
Log-normal & -0.5 & 100 & 0.936 & 0.003 & 6.558 & 0.000 \\
Log-normal & 0.3 & 30 & 0.898 & 0.004 & 8.587 & 0.000 \\
Log-normal & 0.3 & 60 & 0.918 & 0.004 & 7.165 & 0.000 \\
Log-normal & 0.3 & 100 & 0.931 & 0.004 & 6.634 & 0.000 \\
Log-normal & 0.8 & 30 & 0.893 & 0.004 & 8.849 & 0.000 \\
Log-normal & 0.8 & 60 & 0.922 & 0.004 & 7.071 & 0.000 \\
Log-normal & 0.8 & 100 & 0.931 & 0.004 & 6.638 & 0.000 \\
Normal & -0.5 & 30 & 0.924 & 0.004 & 7.076 & 0.001 \\
Normal & -0.5 & 60 & 0.942 & 0.003 & 6.421 & 0.000 \\
Normal & -0.5 & 100 & 0.949 & 0.003 & 6.071 & 0.000 \\
Normal & 0.3 & 30 & 0.926 & 0.004 & 7.048 & 0.000 \\
Normal & 0.3 & 60 & 0.942 & 0.003 & 6.368 & 0.000 \\
Normal & 0.3 & 100 & 0.942 & 0.003 & 6.293 & 0.000 \\
Normal & 0.8 & 30 & 0.924 & 0.004 & 7.078 & 0.000 \\
Normal & 0.8 & 60 & 0.941 & 0.003 & 6.297 & 0.000 \\
Normal & 0.8 & 100 & 0.948 & 0.003 & 6.071 & 0.000 \\
Student $t_5$ & -0.5 & 30 & 0.913 & 0.004 & 7.562 & 0.003 \\
Student $t_5$ & -0.5 & 60 & 0.933 & 0.004 & 6.683 & 0.000 \\
Student $t_5$ & -0.5 & 100 & 0.939 & 0.003 & 6.461 & 0.000 \\
Student $t_5$ & 0.3 & 30 & 0.912 & 0.004 & 7.588 & 0.000 \\
Student $t_5$ & 0.3 & 60 & 0.933 & 0.004 & 6.583 & 0.000 \\
Student $t_5$ & 0.3 & 100 & 0.943 & 0.003 & 6.262 & 0.000 \\
Student $t_5$ & 0.8 & 30 & 0.912 & 0.004 & 7.437 & 0.000 \\
Student $t_5$ & 0.8 & 60 & 0.929 & 0.004 & 6.919 & 0.000 \\
Student $t_5$ & 0.8 & 100 & 0.939 & 0.003 & 6.355 & 0.000 \\
\end{longtable}

\begin{longtable}{llrrrrr}
\caption{Fixed-point boundary results for $\mu_0=(0,0)^\top$. The chi-bar and ordinary critical values are 5.138 and 5.991.}\label{tab:boundary-full}\\
\toprule
Model & $\rho$ & $n$ & Size, $\bar\chi^2_{1,2}$ & MCSE & Size, $\chi^2_2$ & Boundary rate \\
\midrule
\endfirsthead
\toprule
Model & $\rho$ & $n$ & Size, $\bar\chi^2_{1,2}$ & MCSE & Size, $\chi^2_2$ & Boundary rate \\
\midrule
\endhead
\midrule
\multicolumn{7}{r}{Continued on next page}\\
\endfoot
\bottomrule
\endlastfoot
Log-normal & -0.5 & 30 & 0.104 & 0.004 & 0.081 & 0.497 \\
Log-normal & -0.5 & 60 & 0.067 & 0.004 & 0.048 & 0.502 \\
Log-normal & -0.5 & 100 & 0.063 & 0.003 & 0.044 & 0.504 \\
Log-normal & 0.3 & 30 & 0.090 & 0.004 & 0.067 & 0.515 \\
Log-normal & 0.3 & 60 & 0.070 & 0.004 & 0.048 & 0.499 \\
Log-normal & 0.3 & 100 & 0.060 & 0.003 & 0.042 & 0.497 \\
Log-normal & 0.8 & 30 & 0.108 & 0.004 & 0.084 & 0.493 \\
Log-normal & 0.8 & 60 & 0.081 & 0.004 & 0.057 & 0.503 \\
Log-normal & 0.8 & 100 & 0.057 & 0.003 & 0.040 & 0.511 \\
Normal & -0.5 & 30 & 0.072 & 0.004 & 0.049 & 0.500 \\
Normal & -0.5 & 60 & 0.060 & 0.003 & 0.041 & 0.500 \\
Normal & -0.5 & 100 & 0.053 & 0.003 & 0.032 & 0.505 \\
Normal & 0.3 & 30 & 0.073 & 0.004 & 0.052 & 0.499 \\
Normal & 0.3 & 60 & 0.062 & 0.003 & 0.041 & 0.496 \\
Normal & 0.3 & 100 & 0.058 & 0.003 & 0.037 & 0.493 \\
Normal & 0.8 & 30 & 0.073 & 0.004 & 0.051 & 0.494 \\
Normal & 0.8 & 60 & 0.056 & 0.003 & 0.039 & 0.505 \\
Normal & 0.8 & 100 & 0.055 & 0.003 & 0.035 & 0.498 \\
Student $t_5$ & -0.5 & 30 & 0.086 & 0.004 & 0.064 & 0.502 \\
Student $t_5$ & -0.5 & 60 & 0.072 & 0.004 & 0.049 & 0.497 \\
Student $t_5$ & -0.5 & 100 & 0.060 & 0.003 & 0.041 & 0.496 \\
Student $t_5$ & 0.3 & 30 & 0.084 & 0.004 & 0.065 & 0.506 \\
Student $t_5$ & 0.3 & 60 & 0.062 & 0.003 & 0.045 & 0.500 \\
Student $t_5$ & 0.3 & 100 & 0.063 & 0.003 & 0.045 & 0.501 \\
Student $t_5$ & 0.8 & 30 & 0.074 & 0.004 & 0.054 & 0.498 \\
Student $t_5$ & 0.8 & 60 & 0.063 & 0.003 & 0.043 & 0.505 \\
Student $t_5$ & 0.8 & 100 & 0.067 & 0.004 & 0.046 & 0.498 \\
\end{longtable}

\begin{longtable}{llrrrr}
\caption{Composite equality-null results using the analytic $\bar\chi^2_{0,1}$ critical value 2.706.}\label{tab:composite-full}\\
\toprule
Model & $\rho$ & $n$ & Rejection rate & MCSE & Empirical $q_{.95}$ \\
\midrule
\endfirsthead
\toprule
Model & $\rho$ & $n$ & Rejection rate & MCSE & Empirical $q_{.95}$ \\
\midrule
\endhead
\midrule
\multicolumn{6}{r}{Continued on next page}\\
\endfoot
\bottomrule
\endlastfoot
Log-normal & -0.5 & 30 & 0.062 & 0.003 & 3.058 \\
Log-normal & -0.5 & 60 & 0.055 & 0.003 & 2.876 \\
Log-normal & -0.5 & 100 & 0.053 & 0.003 & 2.773 \\
Log-normal & 0.3 & 30 & 0.070 & 0.004 & 3.417 \\
Log-normal & 0.3 & 60 & 0.062 & 0.003 & 3.088 \\
Log-normal & 0.3 & 100 & 0.063 & 0.003 & 3.108 \\
Log-normal & 0.8 & 30 & 0.069 & 0.004 & 3.455 \\
Log-normal & 0.8 & 60 & 0.065 & 0.003 & 3.140 \\
Log-normal & 0.8 & 100 & 0.056 & 0.003 & 2.870 \\
Normal & -0.5 & 30 & 0.052 & 0.003 & 2.795 \\
Normal & -0.5 & 60 & 0.053 & 0.003 & 2.761 \\
Normal & -0.5 & 100 & 0.051 & 0.003 & 2.738 \\
Normal & 0.3 & 30 & 0.055 & 0.003 & 2.926 \\
Normal & 0.3 & 60 & 0.049 & 0.003 & 2.654 \\
Normal & 0.3 & 100 & 0.052 & 0.003 & 2.737 \\
Normal & 0.8 & 30 & 0.051 & 0.003 & 2.769 \\
Normal & 0.8 & 60 & 0.051 & 0.003 & 2.718 \\
Normal & 0.8 & 100 & 0.045 & 0.003 & 2.550 \\
Student $t_5$ & -0.5 & 30 & 0.056 & 0.003 & 2.933 \\
Student $t_5$ & -0.5 & 60 & 0.059 & 0.003 & 2.982 \\
Student $t_5$ & -0.5 & 100 & 0.052 & 0.003 & 2.752 \\
Student $t_5$ & 0.3 & 30 & 0.068 & 0.004 & 3.294 \\
Student $t_5$ & 0.3 & 60 & 0.057 & 0.003 & 2.934 \\
Student $t_5$ & 0.3 & 100 & 0.051 & 0.003 & 2.749 \\
Student $t_5$ & 0.8 & 30 & 0.063 & 0.003 & 3.029 \\
Student $t_5$ & 0.8 & 60 & 0.057 & 0.003 & 2.974 \\
Student $t_5$ & 0.8 & 100 & 0.056 & 0.003 & 2.892 \\
\end{longtable}

\begin{longtable}{lrrrrrr}
\caption{Local-to-boundary results at $\rho=0.3$, where $\delta_n=\eta\,\mathrm{sd}(X_2-X_1)/\sqrt n$. At $\eta=0$ the boundary critical value is used; for $\eta>0$ the interior $\chi^2_2$ critical value is used.}\label{tab:local-full}\\
\toprule
Model & $n$ & $\eta$ & $\delta_n$ & Coverage & MCSE & Boundary rate \\
\midrule
\endfirsthead
\toprule
Model & $n$ & $\eta$ & $\delta_n$ & Coverage & MCSE & Boundary rate \\
\midrule
\endhead
\midrule
\multicolumn{7}{r}{Continued on next page}\\
\endfoot
\bottomrule
\endlastfoot
Log-normal & 30 & 0.0 & 0.0000 & 0.905 & 0.004 & 0.507 \\
Log-normal & 30 & 0.5 & 0.1108 & 0.912 & 0.004 & 0.300 \\
Log-normal & 30 & 1.0 & 0.2216 & 0.904 & 0.004 & 0.152 \\
Log-normal & 30 & 2.0 & 0.4433 & 0.897 & 0.004 & 0.024 \\
Log-normal & 30 & 4.0 & 0.8865 & 0.888 & 0.004 & 0.000 \\
Log-normal & 60 & 0.0 & 0.0000 & 0.929 & 0.004 & 0.494 \\
Log-normal & 60 & 0.5 & 0.0784 & 0.943 & 0.003 & 0.309 \\
Log-normal & 60 & 1.0 & 0.1567 & 0.931 & 0.004 & 0.151 \\
Log-normal & 60 & 2.0 & 0.3134 & 0.931 & 0.004 & 0.026 \\
Log-normal & 60 & 4.0 & 0.6269 & 0.926 & 0.004 & 0.000 \\
Log-normal & 100 & 0.0 & 0.0000 & 0.936 & 0.003 & 0.493 \\
Log-normal & 100 & 0.5 & 0.0607 & 0.949 & 0.003 & 0.310 \\
Log-normal & 100 & 1.0 & 0.1214 & 0.944 & 0.003 & 0.162 \\
Log-normal & 100 & 2.0 & 0.2428 & 0.932 & 0.004 & 0.022 \\
Log-normal & 100 & 4.0 & 0.4856 & 0.934 & 0.004 & 0.000 \\
Normal & 30 & 0.0 & 0.0000 & 0.928 & 0.004 & 0.496 \\
Normal & 30 & 0.5 & 0.1080 & 0.947 & 0.003 & 0.310 \\
Normal & 30 & 1.0 & 0.2160 & 0.928 & 0.004 & 0.147 \\
Normal & 30 & 2.0 & 0.4320 & 0.926 & 0.004 & 0.022 \\
Normal & 30 & 4.0 & 0.8641 & 0.932 & 0.004 & 0.000 \\
Normal & 60 & 0.0 & 0.0000 & 0.942 & 0.003 & 0.505 \\
Normal & 60 & 0.5 & 0.0764 & 0.948 & 0.003 & 0.307 \\
Normal & 60 & 1.0 & 0.1528 & 0.954 & 0.003 & 0.164 \\
Normal & 60 & 2.0 & 0.3055 & 0.942 & 0.003 & 0.026 \\
Normal & 60 & 4.0 & 0.6110 & 0.938 & 0.003 & 0.000 \\
Normal & 100 & 0.0 & 0.0000 & 0.945 & 0.003 & 0.505 \\
Normal & 100 & 0.5 & 0.0592 & 0.964 & 0.003 & 0.311 \\
Normal & 100 & 1.0 & 0.1183 & 0.958 & 0.003 & 0.158 \\
Normal & 100 & 2.0 & 0.2366 & 0.943 & 0.003 & 0.023 \\
Normal & 100 & 4.0 & 0.4733 & 0.944 & 0.003 & 0.000 \\
Student $t_5$ & 30 & 0.0 & 0.0000 & 0.912 & 0.004 & 0.504 \\
Student $t_5$ & 30 & 0.5 & 0.1080 & 0.936 & 0.003 & 0.300 \\
Student $t_5$ & 30 & 1.0 & 0.2160 & 0.924 & 0.004 & 0.153 \\
Student $t_5$ & 30 & 2.0 & 0.4320 & 0.909 & 0.004 & 0.022 \\
Student $t_5$ & 30 & 4.0 & 0.8641 & 0.914 & 0.004 & 0.001 \\
Student $t_5$ & 60 & 0.0 & 0.0000 & 0.936 & 0.003 & 0.495 \\
Student $t_5$ & 60 & 0.5 & 0.0764 & 0.947 & 0.003 & 0.308 \\
Student $t_5$ & 60 & 1.0 & 0.1528 & 0.946 & 0.003 & 0.158 \\
Student $t_5$ & 60 & 2.0 & 0.3055 & 0.935 & 0.003 & 0.023 \\
Student $t_5$ & 60 & 4.0 & 0.6110 & 0.935 & 0.003 & 0.000 \\
Student $t_5$ & 100 & 0.0 & 0.0000 & 0.938 & 0.003 & 0.504 \\
Student $t_5$ & 100 & 0.5 & 0.0592 & 0.957 & 0.003 & 0.294 \\
Student $t_5$ & 100 & 1.0 & 0.1183 & 0.949 & 0.003 & 0.157 \\
Student $t_5$ & 100 & 2.0 & 0.2366 & 0.942 & 0.003 & 0.022 \\
Student $t_5$ & 100 & 4.0 & 0.4733 & 0.942 & 0.003 & 0.000 \\
\end{longtable}

\end{document}